\documentclass{mcom-l}

\usepackage{lipsum,color}
\usepackage{epstopdf}
\usepackage{amsmath,amssymb}
\usepackage{amsfonts}
\usepackage{algorithm}
\usepackage{algorithmic}
\usepackage{url}
\usepackage{subcaption}
\usepackage{graphicx}
\usepackage{mathtools}
\usepackage{rotating}
\usepackage[breaklinks=true,pdfstartview=FitH]{hyperref}

\ifpdf
  \DeclareGraphicsExtensions{.eps,.pdf,.png,.jpg}
\else
  \DeclareGraphicsExtensions{.eps}
\fi

\def\noprint#1{}

\newtheorem{theorem}{Theorem}[section]
\newtheorem{lemma}[theorem]{Lemma}
\newtheorem{corollary}[theorem]{Corollary}

\theoremstyle{definition}

\theoremstyle{remark}

\numberwithin{equation}{section}

\newcommand{\rhobar}{\bar{\rho}}
\newcommand{\rhohat}{\hat{\rho}}
\newcommand{\eps}{\epsilon}
\newcommand{\Aeps}{A_{\epsilon}}

\newcommand{\bAeps}{\bar{A}_{\epsilon}}
\newcommand{\hAeps}{\hat{A}_{\epsilon}}
\newcommand{\tAeps}{\tilde{A}_{\epsilon}}

\newcommand{\bfone}{\mathbf{1}}

\newcommand{\dav}{d_{\textrm{av}}}
\newcommand{\davt}{d_{\textrm{av,2}}}

\newcommand{\R}{\mathbb{R}}

\def\Lmax{L_{\max}}

\def\Lavg{L_{\text{avg}}}

\def\rhoRPCD{\rho_{\text{RPCD}}}

\def\rhoCCD{\rho_{\text{CCD}}}

\def\rhoRCD{\rho_{\text{RCD}}}

\newcommand{\E}{\mathbb{E}}

\newcommand{\crho}{\rho_1}
\newcommand{\crr}{\mathbf{r_1}}
\newcommand{\cRR}{\mathbf{R_1}}

\newcommand{\ddd}{\delta}

\def\tto{\;{\lower 1pt \hbox{$\rightarrow$}}\kern -10pt
           \hbox{\raise 2.8pt \hbox{$\rightarrow$}}\;}

\newcommand{\trace}{\mbox{\rm trace}}

\newcommand{\TheTitle}{Analyzing Random Permutations for Cyclic
  Coordinate Descent}  

\title{\TheTitle}

\usepackage{amsopn}
\DeclareMathOperator{\diag}{diag}
\DeclareMathOperator{\trans}{(transpose)}

\subjclass[2010]{Primary 65F10; Secondary 90C25, 68W20}
\newcommand{\cpmodify}[1]{#1}
\newcommand{\swmodify}[1]{#1}
\begin{document}
\author{Stephen J. Wright}
\address{Computer Sciences Department, University of Wisconsin-Madison, Madison, WI.}
\email{swright@cs.wisc.edu}
\author{Ching-pei Lee}
\address{Department of Mathematics, National University of Singapore.}
\email{leechingpei@gmail.com}
\thanks{This work was supported by NSF Awards
IIS-1447449, 1628384, 1634597,
and 1740707; ONR Award N00014-13-1-0129; AFOSR Award FA9550-13-1-0138,
and Subcontracts 3F-30222 and 8F-30039 from Argonne National
Laboratory; and Award N660011824020 from the DARPA Lagrange Program.
\cpmodify{Ching-pei Lee was at the University of Wisconsin-Madison when this
work was originally submitted.}}
\date{\today}

\begin{abstract}
We consider coordinate descent methods for minimization of convex
quadratic functions, in which exact line searches are performed at
each iteration. (This algorithm is identical to Gauss-Seidel on the
equivalent symmetric positive definite linear system.) We describe a
class of convex quadratic functions for which the random-permutations
version of cyclic coordinate descent (RPCD) is observed to outperform
the standard cyclic coordinate descent (CCD) approach on computational
tests, yielding convergence behavior similar to the fully-random
variant (RCD). A convergence analysis is developed to explain the
empirical observations.
\end{abstract}
\maketitle

\keywords{Coordinate descent, Gauss-Seidel, randomization, permutations}

\section{Introduction} \label{sec:intro}

The coordinate descent (CD) approach for solving the problem
\begin{equation} \label{eq:f}
\min \, f(x), \quad \mbox{where $f:\R^n \to \R$ is smooth and convex,}
\end{equation} 
follows the framework of Algorithm~\ref{alg:cd}. We denote
\begin{equation} \label{eq:notation}
\nabla_i f(x) = [\nabla f(x)]_i, \quad
e_i=(0,\dotsc,0,1,0,\dotsc,0)^T, 
\end{equation}
where the single nonzero in $e_i$ appears in position $i$.  Epochs
(indicated by the counter $\ell$) encompass cycles of inner iterations
(indicated by $j$). At each iteration $k$, one component of $x$ is
selected for updating; a steplength parameter $\alpha_k$ is applied to
the negative gradient of $f$ with respect to this component.

\begin{algorithm} 
\begin{algorithmic}
\STATE Set Choose $x^0 \in \R^n$;
\FOR{$\ell=0,1,2,\dotsc$}
\FOR{$j=0,1,2,\dotsc,n-1$}
\STATE Define $k=\ell n+j$;
\STATE Choose index $i=i(\ell,j) \in \{1,2,\dotsc,n\}$;
\STATE Choose  $\alpha_k>0$;
\STATE $x^{k+1} \leftarrow x^k - \alpha_k \nabla_{i} f(x^k) e_{i}$;
\ENDFOR
\ENDFOR
\end{algorithmic}
\caption{Coordinate Descent\label{alg:cd}}
\end{algorithm}

The choice of coordinate $i=i(\ell,j)$ to be updated at inner
iteration $j$ of epoch $\ell$ differs between variants of CD, as
follows:
\begin{itemize}
\item For ``cyclic CD'' ({\bf CCD}), we choose $i(\ell,j)=j+1$.
\item For ``fully randomized CD,'' also known as ``stochastic CD,''
  and abbreviated as {\bf RCD}, we choose $i(\ell,j)$ uniformly at
  random from $\{1,2,\dotsc,n\}$ and independently at each
  iteration.
\item For ``random-permutations CD'' (abbreviated as {\bf RPCD}), we
  choose $\pi_{\ell+1}$ at the start of epoch $\ell$ to be a random
  permutation of the index set $\{1,2,\dotsc,n\}$ (chosen uniformly at
  random from the space of random permutations), then set $i(\ell,j)$
  to be the $(j+1)$th entry in $\pi_{\ell+1}$, for
  $j=0,1,2,\dotsc,n-1$.
\end{itemize}
Note that $x^{ln}$ denotes the value of $x$ after $l$ epochs.

We consider in this paper problems in which $f$ is a strictly convex
quadratic, that is
\begin{equation} \label{eq:q}
f(x) = \frac12 x^TAx, 
\end{equation}
with $A$ symmetric positive definite.
Even this restricted class of functions reveals significant diversity
in convergence behavior between the three variants of CD described
above. The minimizer of \eqref{eq:q} is obviously $x^*=0$.
Although \eqref{eq:q} does not contain a linear term, it
is straightforward to extend our results to the case for problems of
the form
\[
	f(x) = \frac12 x^T A x - b^T x
\]
by replacing $x^0$ in several places of our analysis with $x^0 - x^*$,
where $x^* = A^{-1}b$ is the minimizer of this problem.  We assume
that the choice of $\alpha_k$ in Algorithm~\ref{alg:cd} is the {\em
  exact} minimizer of $f$ along the chosen coordinate direction. The
resulting approach is thus equivalent to the Gauss-Seidel method
applied to the linear system $Ax=0$.
The variants CCD, RCD, RPCD can be interpreted as different cyclic /
randomized variants of Gauss-Seidel for this system.

In the RPCD variant, we can express a single epoch as follows.
Letting $P$ be the permutation matrix corresponding to the permutation
$\pi$ on this epoch, we split the symmetrically permuted Hessian into
strictly triangular and diagonal parts as follows:
\begin{equation} \label{eq:split}
P^T A P = L_{P} + \Delta_{P} + L_{P}^T,
\end{equation}
where $L_{P}$ is strictly lower triangular and $\Delta_{P}$ is
diagonal. We then define
\begin{equation} \label{eq:Cl}
C_P := -(L_{P}+\Delta_{P})^{-1} L_{P}^T,
\end{equation}
so that the epoch indexed by $l-1$ can be written as follows:
\begin{equation} \label{eq:xl}
x^{ln} = (P_l  C_{P_l} P_l^T)  x^{(l-1) n},
\end{equation}
where $P_l$ denotes the matrix corresponding to permutation $\pi_l$.
By recursing to the initial point $x^0$, we obtain after $\ell$ epochs
that
\begin{equation} \label{eq:xRGS}
x^{\ell n} = (P_{\ell} C_{P_\ell} P_{\ell}^T) (P_{\ell-1} C_{P_{\ell-1}}
P_{\ell-1}^T) \dotsc (P_1 C_{P_1} P_1^T) x^0,
\end{equation}
yielding a function value of
\begin{equation} \label{eq:ffg}
f(x^{\ell n}) = \frac12 (x^0)^T \left( 
(P_1 C_{P_1}^T P_1^T)  \dotsc (P_{\ell} C_{P_\ell}^T P_{\ell}^T) A
(P_{\ell} C_{P_\ell} P_{\ell}^T) \dotsc (P_1 C_{P_1} P_1^T) \right) x^0.
\end{equation}
We analyze convergence in terms of the expected value of $f$ after
$\ell$ epochs for any given $x^0$, with the expectation taken over the
permutations $P_1,P_2,\dotsc,P_{\ell}$, that is,
\begin{equation} \label{eq:expf}
	\E_{P_1,P_2,\dotsc,P_{\ell}} \, f(x^{\ell n}).
\end{equation}

\subsection{Previous Work} \label{sec:previous}

\cpmodify{ Convergence of RCD is analyzed in \cite{Nes12a}, showing
  that when the objective is strongly convex, the method requires
  $O((n \Lmax / \mu) |\log \hat \epsilon| )$ iterations to reach an
  objective function value that is within $\hat\epsilon$ of the
  optimal value, in expectation, when the coordinates are sampled in a
  uniform random manner, where $\mu$ is the modulus of strong
  convexity and $\Lmax$ is the maximum coordinate-wise Lipschitz
  constant for the gradient. This rate can be improved to $O((n \Lavg
  / \mu) |\log \hat \epsilon| )$ if the sampling probability for each
  coordinate is proportional to the coordinate-wise Lipschitz
  constants, where $\Lavg$ is the average of these constants.  On the
  other hand, the best known convergence rate of CCD for convex
  quadratic problems, given by \cite{SunY16a}, has an iteration
  complexity for reaching an $\hat \epsilon$-accurate solution
  \textit{deterministically} that can be $O(n^2)$ times slower than
  that for RCD in the worst case.  The best worst-case convergence
  guarantees for RPCD so far are still identical to those for
  CCD. (The analyses for CCD assume only that each coordinate is
  processed exactly once per epoch, and are indifferent to the fact
  that the ordering of coordinates can change on each epoch, as in
  RPCD.)  However, in practice it is sometimes observed that RPCD
  behaves in a manner more similar to RCD than CCD (see, for example,
  the experiments in \cite{ShaZ13a} and the talks
  \cite{Wri15f,Wri15g}), and a rigorous explanation for the general
  convergence rate for the expected objective of RPCD over the random
  permutations has been difficult to obtain.  Some trials have been
  conducted to tackle this problem.  Recht and R{\'e} \cite{RecR12a}
  state a conjecture whose consequence is that RPCD converges faster
  than RCD on quadratic problems, but they prove the result only for 
  some special random cases.  Sun et al. \cite{SunLY19a} have analyzed the
  convergence speed of the distance between the expected iterate
  $\E[x^k]$ and the minimizer $x^*$ for convex quadratic problems, but
  this cannot be translated to a result for the expected squared error
  $\E \| x^k-x^*\|^2$ nor the expected function suboptimality $\E
  (f(x^k)-f(x^*))$, which are much more informative
  quantities.\footnote{\swmodify{As an example of why a sequence
      $\{x^k \}$ for which $\E[x^k]=x^*$ does not give useful
      information about convergence rate, consider $x^k=x^*+r^k$,
      where $r^k$ are drawn i.i.d. from $N(0,I)$. Such a sequence has
      $\E[x^k]=x^*$, yet it has $\E \|x^k-x^*\|^2=1$, so cannot be
      said to converge to $x^*$ in expectation.}} }

Computational experience reported in \cite{Wri15f,Wri15g} showed that
for most convex quadratic functions \eqref{eq:q}, the convergence
\swmodify{behaviors of all variants of CD are similar.  For example,
  when $A$ is a matrix of the form $V \Sigma V^T$ where $V$ is random
  orthogonal and $\Sigma$ is a positive diagonal matrix whose
  diagonals (the eigenvalues of $A$) follow a log-uniform
  distribution, then CCD, RCD, and RPCD all converge at roughly the
  same rates, no matter how widely the eigenvalues are
  dispersed. However, these computational tests revealed a class of
  matrices $A$ for which the variants had radically different
  performance: matrices of the form
  \begin{equation} \label{eq:paris}
    A = \ddd V \Sigma V^T + (1-\ddd) \bfone \bfone^T,
  \end{equation}
  for small positive values of $\ddd$, and
  $\bfone=(1,1,\dotsc,1)^T$. For such matrices, the performance of RCD
  and RPCD is similar, but CCD converges much more slowly.  In this
  paper, we explain much of this anomalous behavior by considering a
  matrix closely related to \eqref{eq:paris}, and explaining the
  difference by means of a specialized analysis of RPCD.}

The current
paper is an extension of our paper \cite{LeeW16a} in which, motivated
by the empirical observation above, we \swmodify{considered the
  special case of \eqref{eq:paris} in which $\Sigma=I$, that is,
\begin{equation} \label{eq:Ainvariant}
A := \ddd I + (1-\ddd) \bfone \bfone^T, \quad \mbox{where $\ddd \in
  (0,n/(n-1))$.}
\end{equation}
It was proved by \cite{SunY16a} that this matrix achieves worst-case
convergence behavior for CCD. We showed in \cite{LeeW16a} that a
factor of $O(n^2)$ fewer iterations are required by RPCD to achieve the
same accuracy, and that the complexity of RPCD is similar to RCD in
this case.  }
Salient properties of the matrix \eqref{eq:Ainvariant}
include the following.
\begin{itemize}
\item[(a)] It has eigenvalue $\ddd$ replicated $(n-1)$ times, and a
  single dominant eigenvalue $\ddd+(1-\ddd)n$, and 
\item[(b)] it is invariant under symmetric permutations, that is,
  $P^TAP=A$ for all permutation matrices $P$.
\end{itemize}
The latter property makes the analysis of RPCD much more
straightforward than for more general $A$ of the form
\eqref{eq:paris}. Specifically, it follows from \eqref{eq:Cl} that $C_{P}
\equiv C = -(L+\Delta)^{-1}L^T$, where $A=L+\Delta+L^T$, that is,
$C_P$ is independent of $P$. For the matrix \eqref{eq:Ainvariant}, the
expression \eqref{eq:xl} thus simplifies to
\[
x^{ln} = (P_l C P_l^T) x^{(l-1)n}, \quad  l=1,2,3,\dotsc.
\]

We refer to \cite{LeeW16a} for a more extensive discussion of prior
related work on variants of coordinate descent.  \cpmodify{We note in
  particular that for general convex functions $f$, CCD has weaker
  convergence guarantees than for convex {\em quadratic} $f$, as
  analyzed in \cite{BecT12a,SunH15a,LiZ16a}. By contrast, the
  convergence results for RCD presented in \cite{Nes12a} show no
  difference between quadratic and nonquadratic convex functions.}

\subsection{Contributions}
\label{sec:contribution}

In this work, we study the behavior of the RPCD variant of CD on
problems of the form \eqref{eq:q}, where the coefficient matrix has
the form
\begin{equation} \label{eq:Agen1}
B_u := \ddd I + (1-\ddd) uu^T,
\end{equation}
for some $u \in \R^n$.  This paper focuses on the case in which the
components of $u$ are not too different in magnitude, and are all
close to $1$.
Rather than working directly with \eqref{eq:Agen1}, we work with a
diagonally scaled version that has a form more tractable for analysis.
By scaling \eqref{eq:Agen1} symmetrically with the matrix $U =
\diag(u)$, we obtain
\begin{align}
\nonumber
& \Aeps := \ddd I + (1-\ddd) \bfone \bfone^T + \eps D, \\
\label{eq:AD}
 & \mbox{where $\ddd \in (0,n/(n-1))$, $\eps \ge 0$,}  \\
\nonumber
& \mbox{$D = \diag(d)$,  with $\min_i d_i=0$ and $\max_i d_i=1$.}
\end{align}
(Details are given in Section~\ref{sec:gen}.) \swmodify{Note that both
  forms \eqref{eq:Agen1} and \eqref{eq:AD} are generalizations of
  \eqref{eq:Ainvariant}. They are both closely related to the more
  general form \eqref{eq:paris}, in that \eqref{eq:Agen1} can be
  obtained from \eqref{eq:paris} by a symmetric scaling with
  $\Sigma^{-1/2} V$, while \eqref{eq:AD} has the form \eqref{eq:paris}
  with $V=I$ and $\Sigma = I + (\eps/\ddd)D$. Thus, this paper
  provides a significantly more complete explanation of the anomalous
  convergence behavior involving matrices \eqref{eq:paris} than our
  earlier work.}

For matrices of the form \eqref{eq:AD} in \eqref{eq:q}, this paper
proves similar convergence behavior for RPCD to what was proved in
\cite{LeeW16a} for the special case \eqref{eq:Ainvariant}, in the
regime defined by the following values of the parameters $n$, $\eps$,
and $\ddd$:
\begin{equation}
\label{eq:epsdd}
0 < \ddd  \le \eps, \quad \left|\crho\right| \eps^2 <  \ddd \ll 1, \quad n\eps \le 1.
\end{equation}
where $\crho$ is a positive or negative quantity of modest size,
magnitude not much larger than $1$, and independent of $n$, $\eps$,
and $\ddd$.  We prove that the convergence rate guarantee of RPCD for
problems defined by \eqref{eq:AD} is similar to that of RCD, and much
better than \swmodify{the rate bound for} CCD.  Specifically, we
explain via analysis of a linear recurrence that captures the
epoch-wise behavior of RPCD that the per-epoch objective improvement
is bounded by a factor of approximately
\begin{equation} \label{eq:SR}
1 - 1.4\ddd,
\end{equation}
which is similar to the corresponding factors of approximately
$1-\ddd$ and $1-2\ddd$ that are known for RCD (by different analyses),
and significantly better than the factor of approximately
$(1-\ddd/n^2)$ arising from worst-case theoretical guarantees for CCD.
By the generalization \swmodify{of \eqref{eq:Ainvariant}} to
\eqref{eq:AD} (and thus \eqref{eq:Agen1}), we extend our understanding
of the empirical behavior of RPCD, RCD, and CCD described at the
beginning of Section~\ref{sec:previous}.

\subsection{Remainder of the Paper}
In Section~\ref{sec:gen}, we relate matrices of the forms
\eqref{eq:AD} and \eqref{eq:Agen1}, showing that the behavior of CD is
similar on both.  Section~\ref{sec:rpcd} presents our analysis for the
behavior of RPCD on problem \eqref{eq:q}, \eqref{eq:AD}.  In
particular, we define a sequence of matrices $\{\bAeps^{(t)}\}$ such
that given any initial guess $x^0$, the expected value of the
objective $f(x^{tn})$ after the $t$th epoch is $\tfrac12 (x^0)^T
\bAeps^{(t)} x^0$. We then define a sequence of matrices $\{
\hAeps^{(t)} \}$ that dominates $\{ \bAeps^{(t)} \}$, and that can be
parametrized compactly.  We analyze convergence of the sequence $\{
\hAeps^{(t)} \}$ by means of a spectral analysis of the matrix that
relates its parameters at successive values of $t$, and use it to
develop an estimate of the asymptotic per-epoch improvement of the
objective $f(x^{tn})$, $t=0,1,2,\dotsc$.  We provide an explanation in
Section~\ref{sec:first} for the large decrease in $f$ that is often
observed in the very first iteration of CD, a phenomenon that is not
explained by the asymptotic analysis.  Section~\ref{sec:ccd} discusses
RCD and CCD variants for the problem \eqref{eq:q}, \eqref{eq:AD},
while Section~\ref{sec:comp} reports computational experience with the
three variants.

\subsection{Notation} \label{sec:notation}
In addition to the notation $\crho$ mentioned above, which denotes a
scalar quantity of size not much greater than $1$ and independent of
$n$, $\eps$, and $\ddd$, we make extensive use of vector quantities
$\crr \in \R^n$ and matrix quantities $\cRR \in \R^{n \times n}$
(symmetric in some contexts and nonsymmetric in others), which we
assume are both bounded in norm by $1$, that is,
\begin{equation} \label{eq:crr.def}
  \| \crr\| \le 1, \quad \| \cRR \| \le 1.
\end{equation}
In the case in which $\cRR$ is also symmetric, it follows from these
assumptions that $-I \preceq \cRR \preceq I$.  This notation is
essential to capturing remainder terms that appear in our analysis. In
particular, it allows us to keep explicit track of dependence of the
remainder terms on $n$, $\eps$, and $\ddd$. For example, a vector
quantity whose size is bounded by a modest multiple of $\eps^2 n^{-1}$
can be represented by $\crho \eps^2 n^{-1} \crr$.  The following
estimate follows immediately from this notation:
\begin{equation} 
  \label{eq:crr}
  \cRR v \bfone^T  = \crho \crr \bfone^T \;\; \mbox{provided $\|v \|  \le \crho$}.
\end{equation}

Matrix and vector norms $\| \cdot \|$ signify $\| \cdot \|_2$
throughout, unless some other subscript is specified.

\section{Quadratic functions with Hessians of the form \eqref{eq:AD}}
\label{sec:gen}

We discuss here the matrix of the form \eqref{eq:AD}, explaining its
relationship to \eqref{eq:Agen1} and to \eqref{eq:Ainvariant}, and
giving some preliminaries for the analysis of RPCD on the
corresponding quadratic function.

\subsection{Relating \eqref{eq:AD} to \eqref{eq:Agen1}}

Given $\eps>0$ and $\ddd \in (0,1)$, suppose that $u \in \R^n$ satisfies
\begin{equation} \label{eq:u.bounds}
\min_{i=1,2,\dotsc,n} \, |u_i| = \sqrt\frac{\ddd}{\ddd+\eps}, \quad
\max_{i=1,2,\dotsc,n} \, |u_i| =1.
\end{equation}
Consider the matrix $B_u$ from \eqref{eq:Agen1}.
Defining $U := \diag (u)$, we have
\begin{equation} \label{eq:A.1}
A_{\eps} := U^{-1} B_{u} U^{-1} = \ddd U^{-2} +  (1-\delta) \bfone \bfone^T,
\end{equation}
and note that the diagonal elements of $U^{-2}$ are in the range
$[1,\eps/\ddd+1]$. Thus we can write $\ddd U^{-2} = \ddd I + \eps D$,
where $D$ is diagonal with elements in $[0,1]$, so in fact $A_{\eps}$
in \eqref{eq:A.1} has the form \eqref{eq:AD}.



We verify in Appendix~\ref{app:cdinv} that the iterates generated by
Algorithm~\ref{alg:cd} for a given sequence of indices $i(\ell,j)$ to
\eqref{eq:q} with $A=B_{u}$ from \eqref{eq:Agen1}, and with starting
point $\tilde{x}^0$ and exact line search are isomorphic to the
iterates generated by applying the same algorithm with the same index
sequence to \eqref{eq:q} with $A=\Aeps$ from \eqref{eq:A.1}, with
starting point $x^0 = U \tilde{x}^0$. Specifically, we have $x^k = U
\tilde{x}^k$ for all $k \ge 0$, where $\{\tilde{x}^k\}$ is the iterate
sequence corresponding to \eqref{eq:Agen1} and $\{ x^k \}$ is the
sequence corresponding to \eqref{eq:A.1}. Note that the function
values coincide at each iteration, that is,
\begin{equation} \label{eq:U.4}
\frac12 (\tilde{x}^k)^T  B_{u} \tilde{x}^k =  \frac12 (x^k)^T \Aeps x^k, \quad
k=0,1,2,\dotsc.
\end{equation}
Thus we expect to see similar asymptotic behavior for the quadratic
objectives based on matrices \eqref{eq:Agen1} and \eqref{eq:AD}, from
starting points with the same distribution.

 We note too that the matrix $\Aeps$ from \eqref{eq:AD} is
 ``sandwiched'' between scalar multiples of two matrices of the form
 \eqref{eq:Ainvariant}.  We have
\begin{equation} \label{eq:sand.1}
\ddd I + (1-\ddd) \bfone \bfone^T \le \Aeps \le (1+\eps) \left( \ddd'
I + (1-\ddd') \bfone \bfone^T \right),
\end{equation}
where $\ddd' = (\ddd+\eps)/(1+\eps)$ and ``$\le$'' denotes
element-wise inequality. This observation suggests similar behavior
for RPCD to that proved for the matrices \eqref{eq:Ainvariant} in
\cite{LeeW16a}.  Indeed, we observe similar behavior empirically, but
we could not find a way to exploit the relationship \eqref{eq:sand.1}
in our convergence analysis.  The distinctiveness of the components of
$D$ plays a key role; the effects of $D$ in \eqref{eq:AD} persist
through the epochs.  The analysis techniques in \cite{LeeW16a} make
strong use of the fact that the epoch-wise iteration matrix $C_P$ defined
in \eqref{eq:Cl} is independent of $P$, a fact that no longer holds
for matrices \eqref{eq:AD}.

Representative numerical results for the three versions of CD on
quadratics with Hessians of the form \eqref{eq:Ainvariant} are shown
in Figure~\ref{fig:onebig_actual}. We note here the nearly identical
linear rates of the RPCD and RCD variants, and the much slower rate of
the CCD variant. The same pattern is observed in Figure~\ref{fig:2},
which considers matrices of the forms \eqref{eq:Agen1} and
\eqref{eq:AD}. Note in particular that the latter two matrices are
indistinguishable in their empirical behavior, further justifying our
focus on the form \eqref{eq:AD} in our analysis.

\begin{figure}\centering
\includegraphics[width=0.7\linewidth]{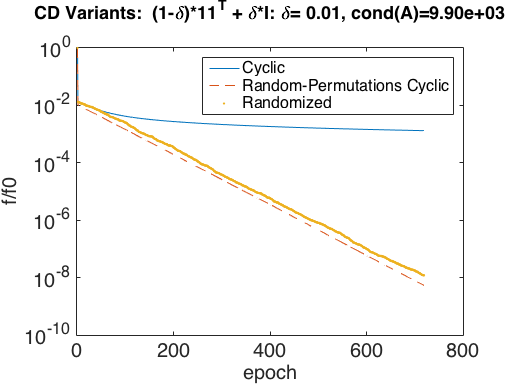}
\caption{CCD, RPCD, and RCD on convex quadratic objective, where $A$
  is defined by \eqref{eq:Ainvariant} with $n=100$ and
  $\ddd=.01$.\label{fig:onebig_actual}}
\end{figure}

\begin{figure}[ht]
	\centering
\begin{subfigure}[b]{0.45\textwidth}
\includegraphics[width=\linewidth]{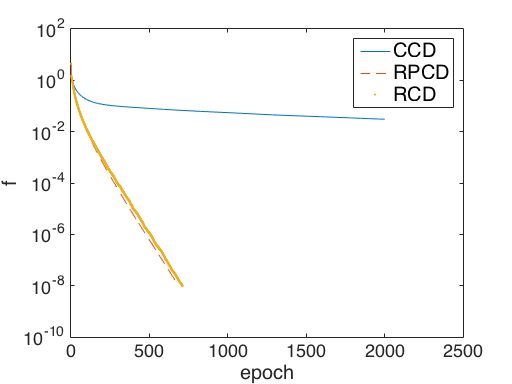}
\caption{Matrix \eqref{eq:AD}.}
\label{fig:AD}
\end{subfigure}
\quad\quad
\begin{subfigure}[b]{0.45\textwidth}
\includegraphics[width=\linewidth]{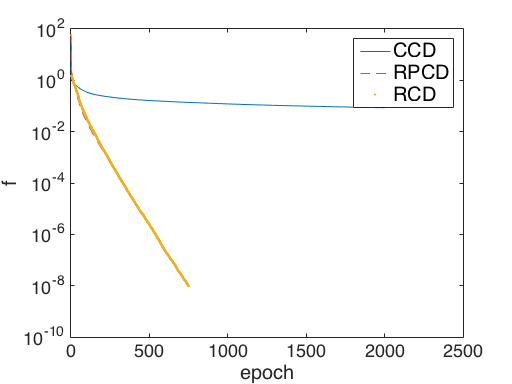}
\caption{Matrix \eqref{eq:Agen1} with $u$ satisfying \eqref{eq:u.bounds}.}
\label{fig:B1}
\end{subfigure}
\caption{Comparison between CCD, RPCD, and RCD on different matrices
with $n=100$ and $(\ddd,\epsilon) = (.01, .05)$.}
\label{fig:2}
\end{figure}

\subsection{RPCD Preliminaries} \label{sec:prelim}

We now define some notation to be used in the remainder of the
analysis: the matrix $C_P$ that defines the change in iterate $x$ over
one epoch and the matrix $\bAeps^{(\ell)}$ that defines the value
$f(x^{\ell n})$ of the objective after $\ell$ epochs.

Applying to \eqref{eq:AD} the decomposition \eqref{eq:split} into
triangular and diagonal matrices, we obtain
\begin{align} 
\nonumber
P^T \Aeps P                   & =  (1-\ddd) E + P^T (\ddd I + \eps D) P  + (1-\ddd) E^T \\
\label{eq:split.e}
                              & = (1-\ddd) E + (\ddd I + \eps D_P) + (1-\ddd) E^T,
\end{align}
where
\begin{equation} \label{eq:def.DE}
D_P := P^TDP, \quad
 E := \left[ \begin{matrix} 0 & 0      & 0      & \dotsc & 0      & 0                   \\
1                             & 0      & 0      & \dotsc & 0      & 0                   \\
1                             & 1      & 0      & \dotsc & 0      & 0                   \\
\vdots                        & \vdots & \vdots &        & \vdots & \vdots              \\
1                             & 1      & 1      & \dotsc & 1      & 0
\end{matrix} \right].
\end{equation}
Following \eqref{eq:Cl}, we have for $\Aeps$ that the epoch matrix
is
\begin{equation} \label{eq:CP}
C_P := - (1-\ddd) \left[ (1-\ddd)E + (I + \eps D_P)  \right]^{-1} E^T.
\end{equation}
Our interest is in the quantity
\begin{equation} \label{eq:expfe}
\E_{P_1,P_2,\dotsc,P_{\ell}} \, f(x^{\ell n}), \quad \ell=1,2,\dotsc,
\end{equation}
where $f(x^{\ell n})$ is defined by \eqref{eq:ffg}.  Adapting notation
from \cite{LeeW16a},  we define the matrices $\bAeps^{(t)}$,
$t=0,1,2,\dotsc, \ell$ as follows:
\begin{align*}
\bAeps^{(0)} &= \Aeps, \\
\bAeps^{(1)} &= \E_{P_{\ell}} \left( (P_{\ell} C_{P_\ell}^T P_{\ell}^T) \Aeps
(P_{\ell} C_{P_\ell} P_{\ell}^T)  \right), \\
\vdots & \\
\bAeps^{(\ell)} &= \E_{P_1,\dotsc,P_{\ell}} \left(
(P_1 C_{P_1}^T P_1^T) \dotsc (P_{\ell} C_{P_\ell}^T P_{\ell}^T) \Aeps
(P_{\ell} C_{P_\ell} P_{\ell}^T) \dotsc (P_1 C_{P_1} P_1^T) \right).
\end{align*}
We have the following recursive relationship between successive
terms in the sequence $\bAeps^{(t)}$, $t=0,1,2,\dotsc$:
\begin{align} 
\label{eq:recursive.t}
\bAeps^{(t)} & = \E_{P_{\ell-t+1}} ( P_{\ell-t+1} C_{P_{\ell-t+1}}^T P_{\ell-t+1}^T \bAeps^{(t-1)}
P_{\ell-t+1} C_{P_{\ell-t+1}} P_{\ell-t+1}^T ) \\
\nonumber
             & = \E_{P} ( P C_P^T P^T \bAeps^{(t-1)} P C_P
P^T ),
\end{align}
where we have dropped the subscript on $P_{\ell-t+1}$ in the
second equality, since the permutation matrices for each epoch are
i.i.d.
Using this matrix, we can compute \eqref{eq:expfe} by
\begin{equation} \label{eq:yf8}
\E_{P_1,P_2,\dotsc,P_{\ell}} \, f(x^{\ell n}) = \frac12
\left( x^0 \right)^T \bAeps^{(t)}x^0.
\end{equation}

\section{Epoch-Wise Convergence of Expected Function Value} \label{sec:rpcd}

In this section, we analyze the behavior of the sequence of matrices
$\{ \bAeps^{(t)} \}$ that govern the expected value of the objective
function $f$ after $t$ epochs of RPCD.  By focusing on the operation
\eqref{eq:recursive.t} which tracks the change from one element of
this sequence to the next, we show that this sequence is bounded in
norm by a quantity that decreases to zero at an asymptotic rate
similar to the known rate for the fully-random variant RCD.

We show that the matrix sequence $\{ \bAeps^{(t)} \}$ is
dominated\footnote{Given two symmetric matrices $F$ and $G$, we say
  that $F$ dominates $G$ if $F-G$ is positive semidefinite.} by
another sequence of positive definite matrices $\{ \hAeps^{(t)} \}$
that can be represented as a four-term recurrence
\begin{equation} \label{eq:emu4}
\hAeps^{(t)} = \hat\eta_t I + \hat\nu_t \bfone \bfone^T + \hat\eps_t D
+ \hat\tau_t (\bfone \crr^T + \crr \bfone^T),
\end{equation}
where $\crr$ is a vector such that $\| \crr \| \le 1$ (as defined in
Section~\ref{sec:notation}) and
$(\hat\eta_t,\hat\nu_t,\hat\eps_t,\hat\tau_t)$ is a quadruplet of
scalar coefficients for all $t=0,1,2,\cdots$.  (Note that the
quantities $\crr$ in the final term are generally different for each
$t$.)  We set $\hAeps^{(0)} = \bAeps^{(0)} = A_{\eps}$, with
\begin{equation} \label{eq:emu4.0}
\hat\eta_0 = \ddd, \quad \hat\nu_0 = 1-\ddd, \quad \hat\eps_0 = \eps,
\quad \hat\tau_0=0,
\end{equation}
and define the sequence $\{ \hAeps^{(t)} \}$ so that successive
elements satisfy the same relationship as shown in
\eqref{eq:recursive.t} for $\{ \bAeps^{(t)} \}$, namely
\[
\hAeps^{(t+1)} \succeq \E_{P} ( P C_P^T P^T \hAeps^{(t)} P C_P P^T).
\]
Our analysis consists chiefly of analyzing the convergence to zero of
the sequence of quadruplets $\{
(\hat\eta_t,\hat\nu_t,\hat\eps_t,\hat\tau_t) \}_{t=0,1,2\dotsc}$
corresponding to $\{ \hAeps^{(t)} \}_{1,2,\dotsc}$.

After several definitions and technical results in
Section~\ref{sec:tech}, we derive in Section~\ref{sec:CP} a
tractable representation of the matrix $C_P$ from \eqref{eq:Cl} that
defines the transition between successive elements of the sequences
$\{ \bAeps^{(t)} \}$ and $\{ \hAeps^{(t)} \}$. In
Section~\ref{sec:main}, we examine the effect of the operation of
$C_P$ on each of the four terms in the bounding sequence
\eqref{eq:emu4}. In Section~\ref{sec:4term}, we define the
recurrence that relates successive elements of the sequence $\{
(\hat\eta_t,\hat\nu_t,\hat\eps_t,\hat\tau_t) \}_{t=0,1,2\dotsc}$, and
examine the rate at which this sequence converges to $0$. We show that
the per-epoch rate is bounded by a scalar sequence that converges at a
nearly linear rate of $1-1.4 \ddd$. (Our analysis is conservative; the
true rate, observed in experiments, is often closer to $1-2\ddd$.)

In most of this section, we consider the regime for parameters $n$,
$\eps$, and $\ddd$ defined by \eqref{eq:epsdd}. The inequality $\ddd
\le \eps$ is made mostly for convenience; it implies that we can
replace $\ddd$ by $\eps$ in remainder terms, and it allows wide
divergence in the diagonal elements of the matrix \eqref{eq:AD}. (We
expect that the main convergence results will continue to apply in a
regime in which $0 \le \eps < \ddd$, which indeed is closer to the
matrix \eqref{eq:Ainvariant} studied in \cite{LeeW16a}, which has
constant diagonals, but the remainder terms in the analysis will need
to be handled differently.) In the analysis of
Section~\ref{sec:4term}, we make additional assumptions on $n$,
$\eps$, and $\ddd$.

\subsection{Definitions and Technical Results}
\label{sec:tech}

We start by defining some useful quantities, drawing on
\cite{LeeW16a}, and proving several elementary results. While
technical, these results give an idea of the effects of applying
expectations over permutations to matrices that arise in the
subsequent analysis.

From \eqref{eq:AD} and \eqref{eq:Dd}, we have
\begin{equation} \label{eq:Dd}
  d = D \bfone, \quad \dav := \bfone^Td/n,
  \quad \davt := \frac1n \bfone^T D^2 \bfone.
\end{equation}
From the definition of $D$ in \eqref{eq:AD}, we have $\dav \in (0,1)$
and $\davt \in (0,1)$.
  We use $\pi$ to denote the permutation of $\{1,2,\dotsc,n\}$
  associated with the permutation matrix $P$, so that for any vector
  $u \in \R^n$, we have
\begin{equation} \label{eq:PDP}
P^Tu = \left[ \begin{matrix} u_{\pi(1)} \\ u_{\pi(2)} \\ \vdots
                                        \\ u_{\pi(n)} \end{matrix} \right], \quad
D_P = P^TDP  = \diag (d_{\pi(1)}, d_{\pi(2)}, \dotsc, d_{\pi(n)}).
\end{equation}
We can see immediately that 
\begin{subequations}
\label{eq:Pe1}
\begin{align} 
P \bfone     & = \bfone,                \\
\E_P \, Pe_j & = \frac{1}{n} \bfone, \quad \mbox{for any $j =1,2,\dotsc,n$.}
\end{align}
\end{subequations}
A useful conditional probability is as follows:
\begin{equation} \label{eq:Pe2}
\E_{P \, | \, P_{i1}=1} Pe_2 = \frac{1}{n-1} (1-e_i).
\end{equation}
This claim follows because $Pe_2$ contains $n-1$ zeros and a single
$1$, and the $1$ cannot appear in position $i$ (because $Pe_1=e_i$)
but may appear in any other position with equal likelihood.

A quantity that appears frequently in the analysis is the matrix $F$
defined by
\begin{equation} \label{eq:def.F}
F := \left[ \begin{matrix} 0 & 1      & 0      & 0 & \dotsc & 0 & 0  \\
0                            & 0      & 1      & 0 & \dotsc & 0 & 0  \\
\vdots                       & \vdots & \vdots &   & \vdots & \vdots \\
0                            & 0      & 0      & 0 & \dotsc & 0 & 1  \\
0                            & 0      & 0      & 0 & \dotsc & 0 & 0
\end{matrix} \right],
\end{equation}
that is, the $n \times n$ matrix of all zeros except for $1$ on the
diagonal immediately above the main diagonal. We see immediately that
$\|F\|=1$. Several identities follow:
\begin{equation} \label{eq:F.facts}
F^T e_1 = e_2, \quad
Fe_1=0, \quad
F \bfone = \bfone-e_n.
\end{equation}
We also have
\begin{equation} \label{eq:PFP}
\E_P \, (PFP^T) = \frac{1}{n} (\bfone \bfone^T - I).
\end{equation}
To verify this claim, note that the diagonals of $PFP^T$ are zero for
all permutation matrices $P$, while the off-diagonals are $1$ with
equal probability. Thus the expected value of the $n(n-1)$
off-diagonal elements is obtained by distributing the $n-1$ nonzeros
in $F$ with equal weight among all off-diagonal elements, giving an
expected value of $1/n$ for each of these elements, as in
\eqref{eq:PFP}.

We have the following results about quantities involving $F$.
\begin{lemma} \label{lem:PFPDP}
\begin{subequations} \label{eq:PFPDP}
\begin{align}
\label{eq:PFPDP.1}
PF P^T DPe_1 & = 0, \\
\label{eq:PFPDP.2}
\E_P \, (PF^T P^T DPe_1) & = \frac{1}{n-1} \left[ \dav \bfone -
  \frac{1}{n} d\right].
\end{align}
\end{subequations}
\end{lemma}
\begin{proof}
For \eqref{eq:PFPDP.1}, we see that $P^T DP e_1$ is a multiple of $e_1$,
and that $Fe_1=0$.

For \eqref{eq:PFPDP.2}, we use $\E_i$ to denote the expectation with
respect to index $i$ uniformly distributed over $\{1,2,\dotsc,n\}$,
and recall that $\pi$ denotes the permutation corresponding to $P$. We
have
\begin{alignat*}{2} 
\E_P \,  (PF^T P^T DPe_1) & = \E_P \, d_{\pi(1)} PF^T e_1  \quad\quad\quad &  & 
\mbox{from \eqref{eq:PDP}}                                                                                     \\
                          & = \E_P \, d_{\pi(1)} Pe_2                      &  & \mbox{from \eqref{eq:F.facts}} \\
                          & = \E_i \, d_i \E_{P \, | \, P_{i1}=1} \, Pe_2  &  &                                \\
                          & = \E_i \, d_i \frac{1}{n-1} (\bfone - e_i)     &  & \mbox{from \eqref{eq:Pe2}}     \\
                          & = \frac{1}{n-1} \left[ \dav \bfone - \frac{1}{n} d\right],
\end{alignat*}
as required.
\end{proof}

Finally, we make frequent use of the following trivial result about
the norm of rank-1 matrices: for any vectors $v, w \in \R^n$, we have
\begin{equation} \label{eq:r1norm}
\| vw^T\| = \|v\| \|w\|.
\end{equation}
In particular, we have from $\| \bfone \| = n^{1/2}$ that
\begin{equation} \label{eq:r1norm.1}
  \| \bfone v^T \| = n^{1/2} \| v\|,
\end{equation}
and in particular, using the notation of
Section~\ref{sec:notation}, we have
\begin{equation} \label{eq:r1norm.r1}
  \| \bfone \crr^T \| \le n^{1/2}.
\end{equation}

\subsection{Properties of the Epoch Matrix $C_P$}
\label{sec:CP}

  As in \cite{LeeW16a}, we define
\begin{equation} \label{eq:Lbar}
\bar{L} := -(I+(1-\ddd)E )^{-1}.
\end{equation}
We noted in \cite{LeeW16a} that
\[
\bar{L}_{ij} = \begin{cases}
-1                   & \;\; \mbox{if $i=j$} \\
(1-\ddd)\ddd^{i-j-1} & \;\; \mbox{if $i>j$} \\
0                    & \;\; \mbox{if $i<j$,}
\end{cases}
\]
so by using notation \eqref{eq:def.F}, we have
\begin{equation} \label{eq:Lbar.2}
\bar{L} = -I + F^T + \ddd \cRR.
\end{equation}
We have further from a standard matrix-norm inequality together with
the facts that $\|\bar{L}\|_1 \le 2$ and $\| \bar{L} \|_{\infty} \le
2$ that
\begin{equation} \label{eq:Lnorm}
\|\bar{L}\|  \le \sqrt{ \| \bar{L} \|_1 \| \bar{L}\|_{\infty} } \le 2.
\end{equation}
Moreover, from \cite[Section~2.2]{LeeW16a}, we have
\[
  (\bar{L} E^T)_{ij} = \begin{cases} -\ddd^{i-1} & \;\; \mbox{for $i<j$} \\
    \ddd^{i-j} - \ddd^{i-1} & \;\; \mbox{for $i \ge j$}, \end{cases}
\]
so that
\begin{equation} \label{eq:LET}
  \bar{L} E^T = I - e_1 \bfone^T + \ddd F^T - \ddd e_2 \bfone^T + \crho \ddd^2 (\cRR + \crr \bfone^T).
\end{equation}
(The validity of the remainder term in this expression follows from
the fact that the coefficients of $\delta^2, \delta^3, \dotsc,
\delta^{n-1}$ in $\bar{L} E^T$ all have the form $\cRR + \crr
\bfone^T$, so we can absorb them all into a single term of order
$\ddd^2$ by summation.)


The following lemma provides a useful estimate of the epoch matrix
$C_P$.
\begin{lemma} \label{lem:CP}
Suppose that \eqref{eq:epsdd} holds. Then for $C_P$ defined by
\eqref{eq:Cl} and \eqref{eq:CP}, we have
\begin{align} \label{eq:CPe}
(1-\ddd)^{-1}  C_P & = I-e_1 \bfone^T + \eps(-D_P + F^TD_P)(I-e_1\bfone^T) \\
    \nonumber
    & \qquad + \ddd (F^T - e_2 \bfone^T)
    + \eps^2 (\crho \crr \bfone^T + \crho \cRR) .
\end{align}
\end{lemma}
\begin{proof}
Note first that for a matrix $Y$ with $\|Y \| \le \crho$ and for
$\eps$ satisfying \eqref{eq:epsdd}, we have
\begin{equation} \label{eq:epsY}
(I-\eps Y)^{-1} = I + \eps Y + \eps^2 (I-\eps Y)^{-1} Y^2 = I + \eps Y
+ \crho \eps^2 \cRR.
\end{equation}
From \eqref{eq:CP}, using definition \eqref{eq:Lbar}, we have
\begin{align*}
 (1-\ddd)^{-1}  C_P & = -  \left[ (I+(1-\ddd)E) + \eps D_P  \right]^{-1} E^T \\
  &=  [\bar{L}^{-1} - \eps D_P]^{-1} E^T \\
  &=  [I - \eps \bar{L} D_P]^{-1} (\bar{L} E^T).
\end{align*}
By substituting from \eqref{eq:LET} and \eqref{eq:epsY} (noting that
$\| \bar{L} D_P \| \le \| \bar{L} \| \le 2$ from \eqref{eq:Lnorm}), we
have
\begin{align*}
(1-\ddd)^{-1}  C_P &=  \left[ I + \eps \bar{L} D_P + \crho \eps^2 \cRR \right]
  \left[ I - e_1 \bfone^T + \ddd F^T - \ddd e_2 \bfone^T + \ddd^2 (\crho \cRR + \crho \crr \bfone^T) \right] \\
  &=  \left[ I - e_1 \bfone^T + \eps \bar{L} D_P (I-e_1 \bfone^T) + \ddd (F^T - e_2 \bfone^T) + \eps^2 (\crho \cRR + \crho \crr \bfone^T) \right],
\end{align*}
where we used $\ddd \le \eps$ from \eqref{eq:epsdd} to absorb the term
$\ddd^2 (\crho \cRR + \crho \crr \bfone^T)$. The result follows immediately when
we use \eqref{eq:Lbar.2} to substitute for $\bar{L}$, and again use
$\ddd \le \eps$ together with $\| D_P \| \le 1$ and \eqref{eq:crr} to
absorb the remainder terms.
\end{proof}


\subsection{Single-Epoch Analysis}
\label{sec:main}

In this section we analyze the change in each term in the expression
\eqref{eq:emu4} over a single epoch. We examine in turn the following
terms:
\begin{itemize}
\item the $I$ term: Lemma~\ref{lem:T1},
\item  the $D$ term: Lemma~\ref{lem:T2},
\item the $\bfone \bfone^T$ and $(\crr \bfone^T + \bfone \crr^T)$
  terms: Lemma~\ref{lem:T3}.
\end{itemize}
Proofs of these technical results appear in Appendix~\ref{app:proofs}.


\begin{lemma} \label{lem:T1}
Suppose that \eqref{eq:epsdd} holds. We have
\begin{align*}
 & (1-\ddd)^{-2} \E_P \, (PC_P^TC_PP^T)                                                                               \\
 & = 
\left[ I + \left( 1-\frac{2}{n} \right) \bfone \bfone^T \right]                                                       \\
 & \;\; + \eps \left[ -2 \left( 1+\frac{1}{n} \right) D +
\frac{3n-2}{n(n-1)} (d \bfone^T + \bfone d^T)  - 2 \frac{n}{n-1} \dav \bfone \bfone^T \right]                         \\
& \;\; + \ddd \left( \frac{-2}{n} \right) I  + \eps^2  (\crho \bfone\bfone^T +
\crho \crr \bfone^T + \crho \bfone \crr^T + \crho \cRR) \\
& \preceq (1+\crho \eps^2) I +  (1+\crho \eps^2) \bfone \bfone^T + (\crho \eps n^{-1/2} + \crho \eps^2) (\bfone \crr^T + \crr \bfone^T).
\end{align*}
\end{lemma}

\begin{lemma} \label{lem:T2}
Suppose that \eqref{eq:epsdd} holds. We have
\begin{align*}
&(1-\ddd)^{-2} \E \, (PC_P^TP^TDPC_P P^T) \\
&=  \left[ D + \dav \bfone \bfone^T - \frac{1}{n} (\bfone d^T +
d \bfone^T ) \right] + \ddd\left[
	-\frac{2}{n} D \right]\\
&\;\; + \eps \left[
-2 \left( 1 + \frac1n \right) D^2 - \frac{\dav}{n-1}(\bfone d^T
+d \bfone^T)
- 2 \davt \bfone \bfone^T \right. \\
& \quad\quad\quad \left. + \frac{2}{n}dd^T +
\frac{2n-1}{n(n-1)}\left(\bfone\bfone^T D^2 + D^2 \bfone \bfone^T\right)
\right] \\
  & \;\; + \eps^2 (\crho \bfone \bfone^T + \crho (\crr \bfone^T +  \bfone \crr^T) + \crho \cRR) \\
  & \preceq  D + (2 \eps + \crho \eps^2) I + (\dav+ \crho \eps^2) \bfone \bfone^T  + (\crho n^{-1/2}+ \crho \eps^2) (\crr \bfone^T + \bfone \crr^T).
\end{align*}
\end{lemma}

\begin{lemma} \label{lem:T3}
Suppose that \eqref{eq:epsdd} holds.  For any $v \in \R^n$, we have
\begin{align}
\label{eq:T3.0}
& (1-\ddd)^{-2} \E_P \, ( PC_P^T P^T (\bfone v^T + v \bfone^T) PC_P P^T )                                                                             \\
\nonumber
& = -\eps \left[ \frac{1}{n} (dv^T+ vd^T) - \frac{\bfone^Tv}{n(n-1)} (d \bfone^T + \bfone d^T) + \frac{1}{n(n-1)} (Dv \bfone^T + \bfone v^TD) \right] \\
\nonumber
                      & \;\; - \ddd \left[ \frac{1}{n-1} (\bfone v^T + v
					  \bfone^T) - \frac{2 \bfone^Tv}{n(n-1)} \bfone
				          \bfone^T \right] \\
\nonumber
& \;\; + \eps^2 n^{1/2} \|v\|  (\crho \cRR + \crho (\bfone \crr^T + \crr \bfone^T) + \crho \bfone \bfone^T )
\end{align}
so that
\begin{align}
\label{eq:T3.0.prec}
&  (1-\ddd)^{-2} \E_P \, ( PC_P^T P^T (\bfone v^T + v \bfone^T) PC_P P^T ) \\
  \nonumber
& \quad\quad \preceq \crho \|v\| (\eps n^{-1/2} + \eps^2 n)  I +
\crho \|v\| (\eps n^{-3/2} + \eps^2 n^{1/2})  \bfone \bfone^T.
\end{align}
When $v=\bfone$, we have
\begin{equation}
\label{eq:T3.v1}
(1-\ddd)^{-2} \E_P \, ( PC_P^T P^T (\bfone \bfone^T)
PC_P P^T ) =
\crho \eps^2 \cRR \preceq \crho \eps^2  I.
\end{equation}
\end{lemma}

The following result summarizes Lemmas~\ref{lem:T1},
\ref{lem:T2}, and \ref{lem:T3}, using the assumption $n \eps \le 1$
from \eqref{eq:epsdd} to simplify some terms.
\begin{theorem} \label{th:4th}
  Suppose that \eqref{eq:epsdd} holds. We have
  \begin{subequations}
\begin{align} 
\label{eq:4th.I}
 & (1-\ddd)^{-2} \E_P \, (PC_P^TC_PP^T) \\
\nonumber
& \quad\quad \preceq (1+\crho \eps^2) I + (1+\crho \eps^2) \bfone \bfone^T + \crho \eps n^{-1/2} (\crr \bfone^T + \bfone \crr^T), \\
\label{eq:4th.11}
& (1-\ddd)^{-2} \E_P \, ( PC_P^T P^T  \bfone \bfone^T PC_P P^T ) \\
\nonumber
& \quad\quad \preceq \crho \eps^2 I, \\
\label{eq:4th.D}
&(1-\ddd)^{-2} \E \, (PC_P^TP^TDPC_P P^T) \\
\nonumber
& \quad\quad \preceq D + (2\eps + \crho \eps^2) I + (\dav + \crho \eps^2) \bfone \bfone^T +
(\crho n^{-1/2} + \crho \eps^2) (\crr \bfone^T + \bfone \crr^T) \\
\label{eq:4th.r1.prec}
& (1-\ddd)^{-2} \E_P \, ( PC_P^T P^T (\bfone \crr^T + \crr \bfone^T) PC_P P^T )\\
\nonumber
& \quad\quad \preceq \crho \eps I + \crho \eps n^{-1/2} \bfone \bfone^T.
\end{align}
\end{subequations}
\end{theorem}
\begin{proof}
The first result \eqref{eq:4th.I} follows immediately from
Lemma~\ref{lem:T1} when we note that $\eps^2 = n^{-1/2} (\eps
n^{-1/2}) (\eps n) \le \eps n^{-1/2}$. The bound \eqref{eq:4th.11} is
immediate from \eqref{eq:T3.v1} in Lemma~\ref{lem:T3}.
Lemma~\ref{lem:T2} immediately yields \eqref{eq:4th.D}.
For \eqref{eq:4th.r1.prec},
we use $\eps n \le 1$ and $\eps^2 n^{1/2} = \eps (n \eps) n^{-1/2} \le
\eps n^{-1/2}$ to simplify the coefficients of $I$ and $\bfone
\bfone^T$ in \eqref{eq:T3.0.prec}.
\end{proof}

\subsection{The Four-Term  Recurrence and Convergence Bound for RPCD}
\label{sec:4term}

In this section we discuss the sequence of $n\times n$ symmetric
matrices $\hAeps^{(t)}$ that dominates the sequence $\bAeps^{(t)}$
defined in Section~\ref{sec:prelim}. Using the results of the
previous subsection, together with the four-term parametrization of
$\hAeps^{(t)}$ defined in \eqref{eq:emu4}, we derive a recurrence
relationship for the sequence of quadruplets $\{(\hat\eta_t,
\hat\nu_t, \hat\eps_t, \hat\tau_t)\}_{t=0,1,2,\dotsc}$. By finding the
rate at which this sequence decreases to zero, we derive a bound on
the expected values of $f$ after each epoch of RPCD.

We now show the main result for recurrence of the representation
\eqref{eq:emu4}.
\begin{theorem} \label{th:4R}
  Suppose that \eqref{eq:epsdd} holds. Consider a nonnegative sequence
  of quadruplets $\{ (\hat\eta_t, \hat\nu_t, \hat\eps_t, \hat\tau_t)
  \}_{t=0,1,2,\dotsc}$ satisfying
\begin{equation} \label{eq:4Rinit}
\hat\eta_0 = \ddd, \quad \hat\nu_0 = 1-\ddd, \quad \hat\eps_0 = \eps,
\quad \hat\tau_0=0,
\end{equation}
along with the recurrence
\begin{equation} \label{eq:4R}
\left[ \begin{matrix}
\tilde{\eta}_{t+1} \\
\tilde{\nu}_{t+1} \\
\tilde{\eps}_{t+1} \\
\tilde{\tau}_{t+1}
\end{matrix} \right] =
(1-\ddd)^2 \hat{M}
\left[ \begin{matrix}
\hat{\eta}_t \\
\hat{\nu}_t \\
\hat{\eps}_t \\
\hat{\tau}_t
  \end{matrix} \right], \quad
\left[ \begin{matrix}
\hat{\eta}_{t+1} \\
\hat{\nu}_{t+1} \\
\hat{\eps}_{t+1} \\
\hat{\tau}_{t+1}
  \end{matrix} \right] =
\left[ \begin{matrix}
\max(\tilde{\eta}_{t+1},0) \\
\max(\tilde{\nu}_{t+1},0) \\
\max(\tilde{\eps}_{t+1},0) \\
\max(\tilde{\tau}_{t+1},0)
\end{matrix} \right],
\end{equation}
where 
\begin{equation} \label{eq:Mhat}
\hat{M} = \left[ \begin{matrix}  
1+ \crho \eps^2  
& \crho \eps^2 
& 2\eps + \crho \eps^2 
& \crho \eps \\
1+\crho \eps^2
& 0 
& \dav + \crho \eps^2
& \crho \eps n^{-1/2} \\
0 & 0 & 1 & 0 \\
\crho \eps n^{-1/2} & 0 & \crho n^{-1/2} + \crho \eps^2 & 0
  \end{matrix} \right],
\end{equation}
where each $\crho$ represents a positive quantity not much greater
than $1$ and independent of $n$, $\eps$, and $\ddd$.  Then we have for
$\hAeps^{(t)}$ defined by \eqref{eq:emu4} that $\hAeps^{(t)} \succeq
\bAeps^{(t)}$ for all $t$.
\end{theorem}
\begin{proof}
  By definition, we have $\hAeps^{(0)} \succeq
  \bAeps^{(0)}$. Supposing that $\hAeps^{(t)} \succeq \bAeps^{(t)}$
  for some $t \ge 0$, we have from \eqref{eq:recursive.t} that
  \begin{equation} \label{eq:kv7}
  \E_P ( PC_P^T P^T \hAeps^{(t)} P C_P P^T) \succeq \E_P ( PC_P^T P^T
  \bAeps^{(t)} P C_P P^T) = \bAeps^{(t+1)}.
  \end{equation}
  Analogous to \eqref{eq:emu4}, we define the following matrix,
  parametrized by the coefficients
  $(\tilde\eta_{t+1},\tilde\nu_{t+1},\tilde\eps_{t+1},\tilde\tau_{t+1})$
  defined in \eqref{eq:4R}:
  \begin{equation} \label{eq:emu4t}
\tAeps^{(t+1)} = \tilde\eta_{t+1} I + \tilde\nu_{t+1} \bfone \bfone^T + \tilde\eps_{t+1} D
+ \tilde\tau_{t+1} (\bfone \crr^T + \crr \bfone^T).
  \end{equation}
Since
$(\hat\eta_{t},\hat\nu_{t},\hat\eps_{t},\hat\tau_{t})
\ge 0$, we can use Theorem~\ref{th:4th} to ensure that
\begin{equation} \label{eq:kv8}
\E_P ( PC_P^T P^T \hAeps^{(t)} P C_P P^T) \preceq \tAeps^{(t+1)}.  
\end{equation}
A little more explanation is needed here. Because the matrices $I$,
$\bfone\bfone^T$, and $D$ (the coefficients of $\hat\eta_{t}$,
$\hat\nu_{t}$, and $\hat\eps_{t}$, respectively) are positive
semidefinite, we can use the upper bounds in \eqref{eq:4th.I},
\eqref{eq:4th.11}, and \eqref{eq:4th.D} to derive the $\preceq$
relationship. The coefficient of $\hat\tau_{t}$ may not be positive
definite, but since $\hat\tau_t \ge 0$, we can still use the bound
\eqref{eq:4th.r1.prec} to establish the $\preceq$
relationship. Moreover, we can assume that $\tilde\tau_{t+1} \ge 0$,
by replacing $\crr$ by $-\crr$ in the representation \eqref{eq:emu4t}
if necessary. Thus, from \eqref{eq:4R}, we have
\[
\tAeps^{(t+1)} - \hAeps^{(t+1)} = \min(\tilde\eta_{t+1},0) I + \min(\tilde\nu_{t+1},0) \bfone\bfone^T + \min(\tilde\eps_{t+1},0) D \preceq 0.
\]
By combining this expression with \eqref{eq:kv7} and \eqref{eq:kv8},
we obtain $\bAeps^{(t+1)} \preceq \hAeps^{(t+1)}$, as required.
\end{proof}

We now analyze the decay of the sequence of quadruplets
$(\hat\eta_t,\hat\nu_t,\hat\eps_t,\hat\tau_t)$ generated by this
recursion.  For purposes of this analysis, we assume that all
quantities $\crho$ that appear in the matrix $\hat{M}$ defined by
\eqref{eq:Mhat} are bounded in magnitude by constant $\rhobar$. From
this constant, we define
\begin{equation} \label{eq:rhohat}
  \rhohat := 3.05 + 2.1 \rhobar + .6 \rhobar^2 + .01 \rhobar^3.
\end{equation}
We place further restrictions on the allowable regime for values of
$n$, $\eps$, and $\ddd$, in addition to those in
\eqref{eq:epsdd}. Specifically, we require
\begin{equation} \label{eq:hs1}
  \rhohat \eps^2 \le \frac12 \ddd, \quad n \ge 5.
\end{equation}
As immediate consequences of these bounds, in combination with
\eqref{eq:epsdd} and \eqref{eq:rhohat}, we have
\begin{subequations} \label{eq:hs2}
  \begin{align}
    &     \rhohat \eps \le \frac12 \frac{\ddd}{\eps} \le \frac12,  \\
        \label{eq:hs2.2}
    &    \eps \le \frac{1}{n} \le .2 \Rightarrow  \ddd \le \eps \le .2 \\
    & n^{-1/2} \le .5, \quad n^{-3/2} \le .1, \quad n^{-2} \le .04, \\
    & \rhobar\eps^2 \le \frac12 \rhohat \eps^2 \le \frac14 \ddd \le .05, \\
    &   \eps^2 =  \frac{(n\eps)^2}{n^2} \le  \frac{1}{n^2} \le .04.
  \end{align}
\end{subequations}
Other useful consequences of \eqref{eq:epsdd} and \eqref{eq:hs1}, used
repeatedly below, are as follows
\begin{subequations} \label{eq:148}
  \begin{alignat}{2}
    \label{eq:1.4}
  (1-\ddd)^2 (1+\rhohat \eps^2) & \le (1-\ddd)^2 (1+\tfrac12 \ddd) && \le (1-1.4\ddd), \\
 \label{eq:1.8}
 (1-\ddd)^2 & = (1-2\ddd+\ddd^2) \le (1-2\ddd+ \ddd/n) && \le (1-1.8 \ddd).
 \end{alignat}
\end{subequations}

We now define two sequences that can be used to bound in norm the
quadruplets $(\hat\eta_t,\hat\nu_t,\hat\eps_t,\hat\tau_t)$. These are
\begin{subequations} \label{eq:etaeps}
  \begin{align}
    \bar\eta_t & := 1.5 \rhohat (1-1.4\ddd)^t t \ddd, \\
      \bar\eps_t & := (1-1.8\ddd)^t \eps.
    \end{align}
\end{subequations}
We note immediately by combining with \eqref{eq:1.4} and
\eqref{eq:1.8} that
\begin{subequations} \label{eq:hs5}
  \begin{align}
  \bar\eta_{t-1} \le \frac{\bar\eta_t}{(1-1.4\ddd)} &\Rightarrow (1-\ddd)^2 \bar\eta_{t-1} \le \bar\eta_t, \\
  \bar\eps_{t-1} = \frac{\bar\eps_t}{(1-1.8\ddd)} &\Rightarrow (1-\ddd)^2 \bar\eps_{t-1} \le \bar\eps_t.
  \end{align}
  \end{subequations}

The following lemma details how the sequence of quadruplets
$(\hat\eta_t,\hat\nu_t,\hat\eps_t,\hat\tau_t)$ is bounded in terms of
the quantities in \eqref{eq:hs5}. Its proof appears in
Appendix~\ref{app:C}.
\begin{lemma} \label{lem:hatbar}
  Assume that the conditions \eqref{eq:epsdd} and \eqref{eq:hs1} hold,
  and let $(\hat\eta_t,\hat\nu_t,\hat\eps_t,\hat\tau_t)$ be defined as
  in Theorem~\ref{th:4R}, and $\bar\eta_t$ and $\bar\eps_t$ be defined
  as in \eqref{eq:etaeps}. Then the following bounds hold for all
  $t=1,2,\dotsc$:
  \begin{subequations} \label{eq:hs6}
    \begin{align}
      \label{eq:hs6a}
      0 \le  \hat\eta_t  & \le \bar\eta_t, \\
      \label{eq:hs6b}
      0 \le \hat\eps_t  & \le \bar\eps_t, \\
      \label{eq:hs6c}
      0 \le \hat\tau_t & \le .5 \eps \rhobar \bar\eta_t + .54 \rhobar \bar\eps_t \\
      \label{eq:hs6d}
      & \le .1 \rhobar \bar\eta_t + .54 \rhobar \bar\eps_t, \\
      \label{eq:hs6e}
      0 \le \hat\nu_t & \le (1.1 + .01 \rhobar^2) \bar\eta_t + (1.1 + .1
      \rhobar^2) \bar\eps_t.
      \end{align}
    \end{subequations}
\end{lemma}

We are now ready to prove the main convergence result.
\begin{theorem} \label{th:conv}
  Suppose that the RPCD version of Algorithm~\ref{alg:cd} is applied
  to function $f$ defined by \eqref{eq:q} with coefficient matrix
  satisfying \eqref{eq:AD}. Suppose that the quantities $\crho$ in
  each recurrence matrix $\hat{M}$ in \eqref{eq:4R} are all bounded in
  magnitude by $\rhobar$, and that conditions \eqref{eq:epsdd} and
  \eqref{eq:hs1} hold. Then there is a constant $C$ such that for all
  $t=1,2,\dotsc$, we have
  \[
  \E_{P_1,P_2, \dotsc, P_t} \, f(x^{tn}) \le C (1-1.4\ddd)^t t \eps \| x^0 \|^2.
  \]
indicating an asymptotic per-epoch convergence rate approaching $1-1.4
\ddd$.
\end{theorem}
\begin{proof}
  The proof follows from \eqref{eq:yf8} and \eqref{eq:etaeps} when we
  use Lemma~\ref{lem:hatbar}, the bound $\| \hAeps^{(t)} \| \le
  \bar{C} \|(\hat\eta_t, \hat\nu_t, \hat\eps_t, \hat\tau_t) \|$ for
  some $\bar{C}>0$, and the bound
  \[
  \|(\hat\eta_t, \hat\nu_t, \hat\eps_t, \hat\tau_t) \| \le \hat{C}
  \max(\bar\eta_t,\bar\eps_t) \le \hat{C}  (1-1.4\ddd)^t t \eps,
  \]
for some $\hat{C}>0$, where we used $\ddd \le \eps$ in the last
step. The final claim follows by taking the ratio of the bound after
$t+1$ and $t$ epochs, which approaches $(1-1.4\ddd)$ as $t \to
\infty$.
  \end{proof}

\subsection{Decrease in the First Iteration}
\label{sec:first}
A behavior of all CD variants that we observe in
Figures~\ref{fig:onebig_actual}-\ref{fig:2} is that the objective
value decreases dramatically in the very first iteration of the
algorithm.  The theorem below shows that this phenomenon can be
explained for both RCD and RPCD by using an extension of the analysis
in \cite[Theorem~3.4]{LeeW16a}. (Similar reasoning also applies for
CCD in most cases, but there is no guarantee, since adversarial
examples consisting of particular choices of $x^0$ can be
constructed.)  Geometrically, the phenomenon is due to the function
\eqref{eq:q}, \eqref{eq:AD} increasing rapidly along just one
direction --- the all-one direction $\bfone$ --- and more gently in
other directions. Thus an exact line search along {\em any} coordinate
search direction will identify a point near the bottom of this
multidimensional ``trench.''

Our result for first-iteration decrease is as follows.
\begin{theorem} \label{th:firstiter}
Consider solving \eqref{eq:q} with the matrix $A=A_{\eps}$ defined in
\eqref{eq:AD}, and $\eps \in (0,1)$, using CCD, RCD, or RPCD with exact
line search.  Then after a single iteration, we have
\begin{equation}
f(x^1) \le \frac12  \sum_{j \neq i} (x^0_j)^2 (\ddd+\eps d_j)
+ \frac{\left( 1 - \ddd \right) \left(  \ddd + \epsilon \right)}{2 (1 +
\epsilon)} \left( \sum_{j \ne i} x^0_j \right)^2,
\label{eq:first}
\end{equation}
where $i=i(0,0)$ is the coordinate chosen for updating in the first
iteration.  When RCD or RPCD is used, we further have that
\begin{align}
	\E_i f(x^1) &\le \frac{\ddd + \epsilon}{2n}  \left( n -
	\frac{\ddd + \epsilon}{1 + \epsilon}
	\right)\|x^0\|^2
	+ \frac{n-2}{n}\frac{(1-\ddd) (\ddd + \epsilon)}{(1 + \epsilon)}
	\left( \bfone^T x^0 \right)^2 .
\label{eq:first.expected}
\end{align}
\end{theorem}
\begin{proof}
Suppose that $i \in \{1,2,\dotsc,n\}$ is the component chosen for
updating in the first iteration, which is chosen uniformly at random
from $\{ 1,2,\dotsc,n\}$ for RPCD and RCD. After a single step of CD,
we have
\begin{align*}
	x^1_i & = x^0_i - \left(x^0_i + (1-\ddd) \sum_{j \ne i}
	\frac{x^0_j}{1 + \epsilon d_i} \right) = -\frac{1-\ddd}{1 +
	\epsilon d_i} \left( \sum_{j \ne i} x^0_j \right); \\
 x^1_j & = x^0_j, \;\; \mbox{for $j \ne i$.}
\end{align*}
Thus, from \eqref{eq:AD}, we have
\begin{align}
\nonumber
f(x^1)
&=
\frac12 \ddd \|x^1 \|^2 + \frac12 (1-\ddd) \left( \sum_{j=1}^n  x^1_j
\right)^2 + \frac12 \epsilon \sum_{j=1}^n d_j (x^1_j)^2 \\
\nonumber
&= \frac12 \ddd \left[ \sum_{j \neq i} (x^0_j)^2 + \left(\frac{1-\ddd}{1
+ \epsilon d_i}\right)^2
\left( \sum_{j \ne i} x^0_j \right)^2 \right] \\
\nonumber
& \quad\quad\quad  + \frac12 (1-\ddd)
\left[ \sum_{j \ne i} x^0_j - \frac{1-\ddd}{1 + \epsilon d_i}
\sum_{j \ne i} x^0_j \right]^2\\
\nonumber
&\quad\quad\quad + \frac12 \epsilon \left[\sum_{j\neq i}
	(x_j^0)^2 d_j + \left(\frac{1-\ddd}{1 + \epsilon d_i}\right)^2
d_i \left(\sum_{j\neq i} x_j^0\right)^2 \right]
\\
\label{eq:first.iter.1}
&= \frac12 \ddd \sum_{j \neq i} (x^0_j)^2  + \frac12 \epsilon \sum_{j\neq i}(x^0_j)^2 d_j\\
\nonumber
&\quad\quad\quad+ \frac{\left( \sum_{j \ne i} x^0_j \right)^2}{2 (1 + \epsilon
d_i)^2}
\left[\ddd \left(1-\ddd\right)^2
+ (1-\ddd) \left(\epsilon d_i + \ddd
\right)^2 + \epsilon d_i \left(1-\ddd
\right)^2\right].
\end{align}
Since $d_i \in [0,1]$, $\ddd \in (0,1)$, and $\eps \in (0,1)$, it can
be shown that
\[
\frac{1}{(1+\eps d_i)^2} \le \frac{1}{(1+\eps)^2}, \quad
\frac{\ddd+\eps d_i}{1+\eps d_i} \le \frac{\ddd+\eps}{1+\eps}, \quad
\frac{d_i}{(1+\eps d_i)^2} \le \frac{1}{(1+\eps)^2}.
\]
Thus by substitution into \eqref{eq:first.iter.1}, we obtain
\begin{align*}
f(x^1) &\le \frac12 \ddd \sum_{j \neq i} (x^0_j)^2  + \frac12 \epsilon \sum_{j\neq i}(x^0_j)^2 d_j\\
&\quad+ \frac{\left( \sum_{j \ne i} x^0_j \right)^2}{2 (1 + \epsilon)^2}
\left[\ddd (1-\ddd)^2
+ (1-\ddd) (\epsilon + \ddd)^2 + \epsilon  \left(1-\ddd
\right)^2\right].
\end{align*}
Further, by noting that
\begin{align*}
\ddd (1-\ddd)^2
+ (1-\ddd) (\epsilon + \ddd)^2 + \epsilon  \left(1-\ddd
\right)^2 & = (1 - \ddd) \left[ (\ddd + \epsilon) ( 1 - \ddd) +
(\epsilon + \ddd)^2\right] \\
& = \left( 1 - \ddd \right) (\ddd + \epsilon) ( 1 + \epsilon),
\end{align*}
the desired result \eqref{eq:first} is obtained.

The result \eqref{eq:first.expected} is then obtained by noting that
$d_i \leq 1$ and that
\begin{align*}
	\E_i \sum_{j\neq i} \left( x^0_j \right)^2 &= \frac{n-1}{n}
	\|x^0\|^2,\\
	\E_i \left(\sum_{j\neq i} x^0_j \right)^2 &= \left( 1 - \frac2n
	\right)\left( \bfone^T x^0 \right)^2 + \frac1n \|x^0\|^2,
\end{align*}
whose derivation can be found in the proof of
\cite[Theorem~3.4]{LeeW16a}.
\end{proof}

We can compare $f(x^1)$ from this theorem with $f(x^0)$ obtained by
substituting into \eqref{eq:AD}, which is
\[
f(x^0) = \frac12 \sum_{i=1}^n (x_i^0)^2 (\ddd  + \eps d_i) + \frac12 (1-\ddd) (\bfone^T x^0)^2.
\]
Note that the second term, which involves $(\bfone^Tx^0)^2$, decreases
by a factor of approximately $(\ddd+\eps)$ in the first iteration,
whereas the first term, which involves $\|x^0\|^2$, does not change
much from its original value, which is typically already small.  For
most starting points, the decrease is dramatic.

\section{Analysis for CCD and RCD} \label{sec:ccd}

The analysis for the RCD variant of coordinate descent for
\eqref{eq:q}, \eqref{eq:AD} follows from the standard analysis
\cite{Nes12a}. The modulus of convexity $\mu$ is $\ddd$, while the
maximum coordinate-wise Lipschitz constant for the gradient $\Lmax$ is
$1+\eps$.  The per-epoch linear rate of expected improvement in $f$
for RCD on $A_{\eps}$ is thus
\begin{equation} \label{eq:rhoRCD}
	\rhoRCD \leq \left(1 - \frac{\ddd}{n(1 + \eps)} \right)^n
        \approx 1 - \ddd + \ddd \eps + O\left(\ddd\eps^2\right),
\end{equation}
yielding a complexity of $O(|\log \hat \epsilon| / (\ddd (1 -
\epsilon)) )$ iterations for reaching an $\hat \epsilon$-accurate
objective.  The (slightly tighter) complexity of RCD from
\cite[Section~4]{Nes12a} improves this epoch bound by approximately a
factor of $2$, to
\begin{equation}
  O \left(|\log\hat \epsilon| \frac{1 + \eps + \ddd}{2\ddd} \right) \quad
  \mbox{iterations.}
	\label{eq:RCD.complexity}
\end{equation}
That is, the per-epoch convergence rate of a bound on $\rhoRCD$ is
approximately $1 - 2 \ddd / (1 + \epsilon + \ddd)$.

It is also shown in \cite{Nes12a} that one can get an improved rate by
non-uniform sampling of the coordinates when the coordinate-wise
Lipschitz constants are not identical.  In particular, for
\eqref{eq:q}, \eqref{eq:AD}, if the probability that the $i$th
coordinate is sampled is proportional to the value of
$(A_\epsilon)_{ii}$, the result in \cite{Nes12a} improves $\rhoRCD$ to
\begin{equation*}
\rhoRCD \leq \left(1 - \frac{\ddd}{n(1 + \dav \eps)} \right)^n \approx 1 -
\ddd + \ddd \dav \eps + O\left(\ddd\dav^2\eps^2\right).
\end{equation*}
Since $\dav \in (0,1)$, this improvement is rather insignificant, and
the rate is still worse than that of \eqref{eq:RCD.complexity}. 
(Whether nonuniform sampling can improve the complexity expression
\eqref{eq:RCD.complexity} is  unknown.)

For CCD,  we note that the iterates have the form
\[
x^{\ell n} = C^{\ell} x^0,
\]
where $C = -(L+\Delta)^{-1} L^T$, where $A = L+\Delta+L^T$ is the
triangular-diagonal splitting of $A$ (that is, $C=C_I$ from
\eqref{eq:split}, \eqref{eq:Cl}). Thus 
\[
f(x^{\ell n}) = \frac12 (x^0)^T (C^{\ell})^T A C^{\ell} x^0,
\]
and the asymptotic behavior of the sequence of function values is
governed by $\| C^{\ell} \|^2$. By Gelfand's formula \cite{Gel41a},
the asymptotic per-epoch decrease factor is thus approximately
$\rho(C)^2$. Proposition~3.1 of \cite{SunY16a} yields an upper bound
on the per-epoch decrease factor. Noting that the largest eigenvalue
of $A_{\eps}$ is bounded above by $n(1 - \ddd) + \ddd + \epsilon$,
their bound is as follows:
\begin{equation} \label{eq:suny}
	\rhoCCD \leq
1-\max \left\{ \frac{\ddd}{n (n(1-\ddd)+\ddd+\epsilon)},
\frac{\ddd}{(n(1-\ddd)+\ddd + \epsilon)^2 (2+\log n/\pi)^2}, \frac{\ddd}{n^2} \right\},
\end{equation}
which is approximately $1-\ddd/n^2$ for the ranges of values of $\ddd$
and $\eps$ of interest in this paper. The implied iteration complexity
guarantee is about a factor of $n^2$ worse than that for
RCD.  In our computational experiments, we compare empirical
observations of CCD convergence rate with  $\rho(C)^2$ rather than
with \eqref{eq:suny}.

Note that the upper bounds for convergence rates of RCD and CCD are
worst-case guarantees.  On the problem class \eqref{eq:AD}, we show
that the convergence rate of RPCD is similar to the bound for RCD, and
we see in the next section that both bounds are quite tight in
practice. The worst-case bounds on CCD are looser, in the sense that
the computational behavior is not quite as poor as these bounds
suggest. Nevertheless, comparison of the worst-case bounds correctly
foreshadows that relative behavior of the different variants on these
problems, seen in Figure~\ref{fig:2}: CCD is much slower than RCD or
RPCD on this class of problems.

\section{Computational Results} \label{sec:comp}

We report here on some experiments with variants of CD on problems of
the form \eqref{eq:q}, \eqref{eq:AD}. Fixing $n=100$, we tried
different settings of $\eps$ and $\ddd$, and ran the three variants
CCD, RCD, and RPCD for many epochs.  Results are reported in Tables
\ref{tbl:AD1} and \ref{tbl:AD2}.  We obtain empirical estimates of the
per-epoch asymptotic convergence rate by geometrically averaging the
rate over the last 10 epochs, tabulating these observations as
$\rhoCCD(\ddd,\text{observed})$, $\rhoRCD(\ddd,\text{observed})$, and
$\rhoRPCD(\ddd,\text{observed})$.  Since we report the difference
between these quantities and $1$ in the tables, larger numbers
correspond to faster rates. (The numbers in the table are reported in
scientific notation, with $a(b)$ representing $a \times 10^b$.)  As
noted in Section~\ref{sec:ccd}, we use $\rho(C)^2$ as the theoretical
bound on the convergence rate for CCD, while we use $\rhoRCD$ from
\cite[Section~4]{Nes12a} (which corresponds to the complexity
\eqref{eq:RCD.complexity}) as the theoretical estimate of the
convergence rate for RCD. For RPCD, we used $2\ddd$ as a ``benchmark''
value, corresponding to a per-epoch rate of $1-2\ddd$, slightly faster
than the $1-1.4\ddd$ rate proved in Section~\ref{sec:rpcd}.

Not all the settings of parameters $n$, $\eps$, and $\ddd$ in these
tables satisfy the conditions \eqref{eq:epsdd}, \eqref{eq:hs1} that
were assumed in our analysis. We mark with an asterisk those entries
for which these conditions are not satisfied. We note that the
benchmark rate of $1-2\ddd$ continues to hold in regimes beyond the
reach of our theory. This accords with the observation that the matrix
$\hat{M}$ defined in \eqref{eq:Mhat} indeed has norm very close to
$1$, so that behavior of the sequence is governed chiefly by the
$(1-\ddd)^2$ factor in the recurrence \eqref{eq:4R}.

These tables confirm that the empirical performance of RCD and RPCD is
quite similar, across a wide range of parameter values, and markedly
faster than CCD. 


\begin{table}
	\centering {\scriptsize
	\begin{tabular}{|r|lllll|}
\hline
$\ddd$  & 1.0000 (-03) & 3.0000 (-03) & 1.0000 (-02) & 3.0000 (-02) & 1.0000 (-01)\\ \hline
$1-\rhoCCD(\ddd,\text{observed})$  & 3.4122 (-04) & 3.3170 (-04) & 3.3527 (-04) & 6.1266 (-04) & 8.1036 (-04)\\
$1-\rho(C)^2$  & 5.9018 (-06) & 1.7170 (-05) & 6.0912 (-05) & 1.9453 (-04) & 7.5546 (-04)\\
\hline
$1 - \rhoRCD(\ddd,\text{observed})$  & 2.6814 (-03) & 5.8265 (-03) & 2.1983 (-02) & 6.8824 (-02) & 1.4427 (-01)\\
$1 - \rhoRCD(\ddd,\text{predicted})$  & 1.9940 (-03) & 5.9466 (-03) & 1.9419 (-02) & 5.5047 (-02) & 1.5364 (-01)\\
\hline
$1 - \rhoRPCD(\ddd,\text{observed})$  & 2.7048 (-03) & 6.3637 (-03) & 2.1723 (-02) & 6.9230 (-02) & 2.0842 (-01)\\
Benchmark $2\ddd$  & 2.0000 (-03) & 6.0000 (-03) & 2.0000 (-02) & 6.0000 (-02)$^*$ & 2.0000 (-01)$^*$ \\ \hline
	\end{tabular}
}
	\caption{Comparison of CCD, RPCD, and RCD on the matrix
		\eqref{eq:AD} with $n=100$ and $\epsilon = \ddd$.}
	\label{tbl:AD1}
\end{table}
\begin{table}
	\centering
{\scriptsize
	\begin{tabular}{|r|lllll|}
\hline
$\ddd$  & 1.0000 (-03) & 3.0000 (-03) & 1.0000 (-02) & 3.0000 (-02) & 1.0000 (-01)\\
\hline
 $1-\rhoCCD(\ddd,\text{observed})$  & 2.2372 (-04) & 3.9800 (-04) & 3.3538 (-04) & 2.8511 (-04) & 7.9319 (-04)\\
 $1-\rho(C)^2$  & 2.7954 (-05) & 5.0165 (-05) & 9.7958 (-05) & 2.3542 (-04) & 7.8096 (-04)\\
\hline
 $1 - \rhoRCD(\ddd,\text{observed})$  & 2.6143 (-03) & 8.6962 (-03) & 1.7869 (-02) & 5.8402 (-02) & 1.4545 (-01)\\
 $1 - \rhoRCD(\ddd,\text{predicted})$  & 1.9763 (-03) & 5.8634 (-03) & 1.9019 (-02) & 5.3824 (-02) & 1.5364 (-01)\\
\hline
 $1 - \rhoRPCD(\ddd,\text{observed})$  & 2.8377 (-03) & 7.1350 (-03) & 2.1157 (-02) & 6.6712 (-02) & 2.0501 (-01)\\
 Benchmark $2\ddd$  & 2.0000 (-03) & 6.0000 (-03)$^*$ & 2.0000 (-02)$^*$ & 6.0000 (-02)$^*$ & 2.0000 (-01)$^*$ \\
\hline
	\end{tabular}
}
	\caption{Comparison of CCD, RPCD, and RCD on the matrix
		\eqref{eq:AD} with $n=100$ and $\epsilon = \sqrt{\ddd/10}$.}
	\label{tbl:AD2}
\end{table}

\section*{Acknowledgments}

We are grateful for the careful reading and penetrating comments of
two referees, which resulted in significant improvement of the
results.

\bibliographystyle{amsplain} \bibliography{cdrefs}

\appendix

\section{Invariance of Coordinate Descent under Diagonal Scaling}
\label{app:cdinv}

Coordinate descent applied to quadratics \eqref{eq:q} with exact line
search at each iterate is invariant under symmetric diagonal scalings
of $A$. For any symmetric positive definite $A$ and nonzero diagonal
$F$, define
\begin{equation}
	\tilde{A} = F^{-1} A F^{-1}.
\end{equation}
Note that $\tilde{A}$ is symmetric positive definite. Consider the
objective functions \eqref{eq:q} defined with Hessians $A$ and
$\tilde{A}$. For a given $x^0$, define $\tilde{x}^0 = F x^0$. The
function values match at these points, that is,
\begin{equation}
(\tilde{x}^0)^T \tilde{A} \tilde{x}^0 = (Fx^0)^T \tilde{A} (Fx^0) =
  (x^0)^T A x^0.
\end{equation}
Considering the iterates generated by Algorithm~\ref{alg:cd} for the
two functions, with $\alpha_k$ defined by exact line searches, and the
same choices of coordinates $i(\ell,j)$ at each iteration. Assume that
$Fx^{t} = \tilde{x}^{t}$ for $t= 1,2,\dotsc, k$.  Suppose that
coordinate $i$ is chosen at iteration $k$, the updates are
\[
x^{k+1} = x^k - \frac{(A x^{t})_i}{A_{ii}} e_i, \quad \tilde{x}^{k+1}
= \tilde{x}^k - \frac{(\tilde{A} \tilde{x}^{t})_i}{\tilde{A}_{ii}}
e_i.
\]
By noting that
\[
(\tilde{A} \tilde{x}^t)_i = F_{ii}^{-1} (Ax^t)_i, \quad \tilde{A}_{ii} = F_{ii}^{-2} A_{ii},
\]
and using the inductive hypothesis, it is easy to verify that
$\tilde{x}^{k+1} = Fx^{k+1}$, as required.

\section{Proofs of Lemmas from Section~\ref{sec:main}}
\label{app:proofs}

\subsection{Proof of Lemma~\ref{lem:T1}}

\begin{proof}
  From Lemma~\ref{lem:CP}, we have
  \begin{subequations}
\begin{align}
\nonumber
& (1-\ddd)^{-2} PC_P^T C_P P^T                                                                                          \\
\label{eq:T1.01}
&= P \left[ (I-\bfone e_1^T) + \eps (I-\bfone e_1^T) (-D_P + D_P F) +
  \ddd (F-\bfone e_2^T) + \eps^2 (\crho \bfone \crr^T + \crho \cRR) \right] \\
\nonumber
& \quad\quad  \left[ (I-e_1 \bfone^T) + \eps(-D_P + F^T D_P) (I-e_1 \bfone^T) +
  \ddd (F^T-e_2 \bfone^T) + \eps^2 (\crho \crr \bfone^T + \crho \cRR) \right] P^T  \\
\label{eq:T1.0}
 & = \Big\{
P(I-\bfone e_1^T)(I-e_1 \bfone^T) P^T                                                                                 \\
\nonumber
 & \quad\quad + \ddd \big[ P (F-\bfone e_2^T)(I - e_1 \bfone^T) P^T +
  P (I-\bfone e_1^T)(F^T-e_2 \bfone^T) P^T \big] \\
\nonumber
 & \quad\quad
 +
\eps P (I-\bfone e_1^T)(-2D_P + D_PF + F^T D_P)(I-e_1 \bfone^T) P^T
\Big\} \\
\nonumber
& \quad\quad + \eps^2 (\crho (\bfone \crr^T + \crr \bfone^T) + \crho \bfone \bfone^T + \crho \cRR).
\end{align}
\end{subequations}
(We give further details on the $\eps^2$ term below.)
For the $O(1)$ term in \eqref{eq:T1.0}, we have from \eqref{eq:Pe1}
and $e_1^Te_1=1$ that
\begin{equation}
\E_P \, \left( P(I-\bfone e_1^T)(I-e_1 \bfone^T) P^T \right)
 = I - \frac{2}{n} \bfone \bfone^T + \bfone \bfone^T  = 
I + \left( 1-\frac{2}{n} \right) \bfone \bfone^T.
\label{eq:T1.1}
\end{equation}
For the first part of the $O(\ddd)$ term, we have from \eqref{eq:Pe1},
\eqref{eq:F.facts}, \eqref{eq:PFP}, and $e_1^Te_2=0$ that
\begin{align}
\nonumber
\E_P \, \left( P (F-\bfone e_2^T) (I-e_1 \bfone^T) P^T \right) 
 & = 
\E_P \, (PFP^T  - P \bfone e_2^T P^T)                                                                                 \\
\nonumber
 & = \E_P \, \left( PFP^T - \bfone  \left(\frac{1}{n} \bfone \right)^T \right)                                        \\
\label{eq:T1.2}
 & = \frac{1}{n} \bfone \bfone^T - \frac{1}{n} I - \frac{1}{n} \bfone \bfone^T =
-\frac{1}{n} I.
\end{align}
(By symmetry, the second part of the $O(\ddd)$ term will also have
expectation $-\frac{1}{n} I$.)

For the $O(\eps)$ term in \eqref{eq:T1.0}, we have from
\eqref{eq:def.DE}, \eqref{eq:F.facts}, \eqref{eq:PFP},
Lemma~\ref{lem:PFPDP}, and the fact that $\E_P e_1^T D_P e_1 = \dav$
that
\begin{align*}
 & \E_P \, \left\{ P (I-\bfone e_1^T)(-2D_P + D_PF + F^T D_P)(I-e_1 \bfone^T) P^T \right\}                                                                   \\
 & = \E_P \, \left\{
(P-\bfone e_1^T) (-2 P^TDP + P^T DPF + F^TP^TDP) (P^T-e_1 \bfone^T) 
\right\}                                                                                                                                                     \\
 & = -2D + \E_P (DPFP^T + PF^T P^TD)                                                                                                                         \\
 & \quad\quad - \E_P \left\{ \bfone e_1^T (-2P^TD+P^TDPFP^T) + (-2DP + PF^TP^TDP) e_1 \bfone^T \right\}                                                      \\
 & \quad\quad + \E_P (-2 e_1^T (P^TDP) e_1) \bfone \bfone^T                                                                                                   \\
 & =-2D + \frac{1}{n} (D (\bfone \bfone^T - I) + (\bfone \bfone^T -
 I)D)  + \frac{2}{n} (\bfone \bfone^TD + D \bfone \bfone^T)                                                          \\
 & \quad\quad - \frac{1}{n-1} \left[ 2 \dav \bfone \bfone^T - \frac{1}{n} \bfone \bfone^TD -  \frac{1}{n} D \bfone \bfone^T \right] - 2 \dav \bfone \bfone^T \\
 & = -2 \left( 1 + \frac{1}{n} \right) D + \left( \frac{3}{n} + \frac{1}{n(n-1)} \right) \left( d \bfone^T + \bfone d^T \right) - 2 \left( 1 + \frac{1}{n-1} \right) \dav \bfone \bfone^T.
\end{align*}
The lower-order terms in the main result follows by substituting this
estimate along with \eqref{eq:T1.1} and \eqref{eq:T1.2} into
\eqref{eq:T1.0}.

We now address the $\eps^2$ term in \eqref{eq:T1.0}. Gathering
together all terms with coefficients $\eps^2$, $\eps\ddd$, and
$\ddd^2$ from  \eqref{eq:T1.01}, we have
\begin{align*}
  &  \eps^2 P (\crho \bfone \crr^T + \crho \cRR_(I-e_1 \bfone^T)P^T + \mbox{(transpose)} \\
  & + \eps^2 P(I-\bfone e_1^T) (-D_P) (I-F)(I-F^T) (-D_P) (I-e_1 \bfone^T) P^T \\
  &+ \eps \ddd P(I-\bfone e_1^T) (-D_P) (I-F)(F^T-e_2 \bfone^T)P^T + \mbox{(transpose)} \\
  & + \ddd^2 P(F-\bfone e_2^T)(F^T - e_2 \bfone^T)P^T.
\end{align*}
The first term in this expression (and its transpose) is clearly
accounted for by the $\eps^2$ term in \eqref{eq:T1.0}. For the other
$\eps^2$ term, and also the $\eps\ddd$ terms, we use the facts that
$\| F \| = 1$ and $\| D_P (I-F) \| \le \| D_P \| (\|I\| + \|F\|) =
\crho$, along with $\ddd \le \eps$, to deduce that these terms too are
accounted for by the $\eps^2$ term in \eqref{eq:T1.0}. From $\|F\|=1$
and $e_2^Te_2=1$, we can say the same too for the coefficient of
$\ddd^2$.

For the final ``$\preceq$'' claim in the lemma, we use the facts
\eqref{eq:crr.def}, $d \bfone^T = \crho n^{1/2} \crr \bfone^T$ (from
\eqref{eq:r1norm}), $\dav \in (0,1]$, and $D \succeq 0$,
\end{proof}

\subsection{Proof of Lemma~\ref{lem:T2}}

\begin{proof} 
  From Lemma~\ref{lem:CP}, we have
  \begin{subequations}
    \begin{align}
      \nonumber
  & (1-\ddd)^{-2} PC_P^T D_P C_P P^T \\
  \label{eq:T2.00}
&= P \left[ (I-\bfone e_1^T) + \eps (I-\bfone e_1^T) (-D_P + D_P F) +
  \ddd (F-\bfone e_2^T) + \eps^2 (\crho \bfone \crr^T + \crho \cRR) \right] D_P \\
\nonumber
& \quad\quad  \left[ (I-e_1 \bfone^T) + \eps(-D_P + F^T D_P) (I-e_1 \bfone^T) +
  \ddd (F^T-e_2 \bfone^T) + \eps^2 (\crho \crr \bfone^T + \crho \cRR) \right] P^T  \\
      \nonumber
&= P(I-\bfone e_1^T)D_P (I-e_1 \bfone^T) P^T \\
      \nonumber
& \quad + \eps P (I-\bfone e_1^T) D_P (-I+F)D_P (I-e_1 \bfone^T) P^T \\
      \nonumber
& \quad + \eps P (I-\bfone e_1^T) D_P (-I+F^T)D_P (I-e_1 \bfone^T) P^T \\
      \nonumber
& \quad + \ddd P (F-\bfone e_2^T) D_P (I-e_1 \bfone^T) P^T \\
      \nonumber
& \quad + \ddd P (I-\bfone e_1^T) D_P (F^T-e_2 \bfone^T) P^T  \\
      \nonumber
& \quad + \eps^2(\crho \bfone \bfone^T + \crho (\crr \bfone^T + \bfone \crr^T) + \crho \cRR) \\
  \label{eq:T2.01}
&= P(I-\bfone e_1^T)D_P (I-e_1 \bfone^T) P^T \\
      \nonumber
& \quad +\eps P (I-\bfone e_1^T) D_P (-2I+F+F^T) D_P (I-e_1 \bfone^T) P^T \\
      \nonumber
& \quad + \ddd P (FD_P - \bfone e_2^T D_P + D_P F^T - D_P e_2 \bfone^T)P^T \\
      \nonumber
& \quad + \eps^2 (\crho \bfone \bfone^T + \crho (\crr \bfone^T + \bfone \crr^T) + \crho \cRR),
\end{align}
\end{subequations}
where we used \eqref{eq:PFPDP.1} from Lemma~\ref{lem:PFPDP} along with
$e_2^T D_P e_1=0$ to simplify the coefficient of $\ddd$. (Further
justification for the form of the $\eps^2$ term appears below.)

For the $O(1)$ term, we have that
\begin{align*}
P ( I - \bfone e_1^T ) D_P ( I - e_1 \bfone^T) P^T 
& = (P-\bfone e_1^T)D_P (P^T-e_1 \bfone^T) \\
&= D - \bfone e_1^T P^TD - DP e_1 \bfone^T + (e_1^T D_P e_1) \bfone \bfone^T. 
\end{align*}
Thus from \eqref{eq:Pe1}, we have by taking expectations over $P$ that
\begin{align*}
 \E_P  (P ( I - \bfone e_1^T ) D_P ( I - e_1 \bfone^T ) P^T ) 
&= D - \frac{1}{n} \bfone \bfone^TD - \frac{1}{n} D \bfone \bfone^T +
\dav \bfone \bfone^T \\
&= D - \frac{1}{n} (\bfone d^T + d \bfone^T) +
\dav \bfone \bfone^T,
\end{align*}
as required.

For the coefficient of $\ddd$, we have 
\begin{align*}
& P (FD_P - \bfone  e_2^T D_P + D_P F^T - D_P e_2 \bfone^T) P^T  \\
& \quad = PFP^T D - \bfone e_2^T P^T D + DP F^T P^T - DP e_2 \bfone^T.
\end{align*}
Taking expectations with respect to $D$, we have from \eqref{eq:Pe1} and 
\eqref{eq:PFP} that
\begin{align*}
& \E_P (PFP^T)D - \bfone \E_P (e_2^T P^T) D + D \E_P (PF^TP^T) - D \E_P (Pe_2) \bfone^T \\
& \quad = \frac1n (\bfone \bfone^T - I) D - \frac1n \bfone \bfone^T D 
+ \frac1n D (\bfone \bfone^T -I) - \frac1n D \bfone \bfone^T \\
& \quad = -\frac2n D.
\end{align*}

For the coefficient of $\eps$, we have
\begin{align}
\nonumber
& P(I-\bfone e_1^T) P^TDP (-2I+F+F^T) P^TDP (I-e_1 \bfone^T) P^T  \\
\nonumber
& \quad = DP(-2I+F+F^T) P^TD \\
\nonumber
& \quad\quad  - \bfone e_1^T P^TDP (-2I+F+F^T) P^TD 
 - DP (-2I+F+F^T) P^TDP e_1 \bfone^T \\
\nonumber
& \quad\quad + \left[ e_1^T P^TDP (-2I+F+F^T) P^TDP e_1 \right] \bfone \bfone^T \\
\nonumber
& \quad = DP(-2I+F+F^T) P^TD \\
\nonumber
& \quad\quad  - \bfone e_1^T P^TDP (-2I+F) P^TD 
- DP (-2I+F^T) P^TDP e_1 \bfone^T \\
\label{eq:61}
& \quad\quad -2  \left[ e_1^T P^TD^2P e_1 \right] \bfone \bfone^T,
\end{align}
where we used \eqref{eq:PFPDP.1} in Lemma~\ref{lem:PFPDP} to eliminate
terms that are multiples of $Fe_1=0$. 

For the first  term in \eqref{eq:61}, we have  from \eqref{eq:PFP} that 
\begin{align}
\nonumber
& \E_P (DP (-2I+F+F^T) P^TD) \\
\nonumber
& \quad = -2D^2 + D \E_P (PFP^T) D + D \E_P (PF^TP^T) D \\
\nonumber
& \quad = -2D^2 + \frac2n D (\bfone \bfone^T - I) D  \\
\label{eq:61.1}
& \quad = -2 \left( 1+\frac1n \right) D^2 + \frac2n dd^T.
\end{align}
For the second term in \eqref{eq:61}, we have
\begin{align}
\nonumber
& - \bfone \E_P (e_1^T P^T DP (-2I + F) P^T D) \\
\nonumber
& \quad = 2 \bfone \E_P (e_1^T P^T) D^2 - \bfone \E_P (e_1^T P^T DPFP^T) D \\
\nonumber
& \quad = \frac2n \bfone \bfone^T D^2 - \frac{1}{n-1} \bfone
\left( \dav \bfone^T - \frac1n \bfone^T D \right) D \\
\nonumber
& \quad = \frac2n \bfone \bfone^T D^2 - \frac{\dav}{n-1} \bfone d^T + \frac{1}{n(n-1)} \bfone \bfone^T D^2 \\
\label{eq:61.2}
& \quad = \frac1n \frac{2n-1}{n-1} \bfone \bfone^T D^2 - \frac{\dav}{n-1} \bfone d^T, 
\end{align}
where we used \eqref{eq:PFPDP.2} from Lemma~\ref{lem:PFPDP} and the
definition of $\dav$ in \eqref{eq:Dd}. The third term in
\eqref{eq:61} is the transpose of this second term.  For the final
term in \eqref{eq:61}, we have 
\begin{equation} \label{eq:61.3}
-2 \E_P (e_1^T P^T D^2 Pe_1)  \bfone \bfone^T  = -2 \davt \bfone \bfone^T.
\end{equation}
By substituting \eqref{eq:61.1}, \eqref{eq:61.2}, and \eqref{eq:61.3}
into \eqref{eq:61}, we obtain the required coefficient of $\eps$.

We return to verifying the form of the $\eps^2$ term in
\eqref{eq:T2.01}. The coefficients of $\eps^2$, $\eps\ddd$, and
$\ddd^2$ terms from \eqref{eq:T2.00} are are follows:
\begin{align*}
  &  \eps^2 P(I-\bfone e_1^T) D_P (\crho \crr \bfone^T + \crho \cRR) P^T + \mbox{(transpose)} \\
  & + \eps^2 P(I-\bfone e_1^T) (-D_P) (I-F) D_P (I-F^T) (-D_P) (I-e_1 \bfone^T) P^T \\
  & + \eps\ddd P(I-\bfone e_1^T) (-D_P) (I-F) D_P (F^T-e_2 \bfone^T) P^T + \mbox{(transpose)} \\
  &+ \ddd^2 P (F-\bfone e_2^T) D_P (F^T - e_2 \bfone^T).
\end{align*}
By making use of the bounds $\|I\| = \|F\| = 1$, $\|D_P \| \le 1$,
$\|e_1\|=\|e_2 \|=1$, and $\ddd \le \eps$, we see that this expression
is accounted for by the coefficient of $\eps^2$ in \eqref{eq:T2.01}.

For the final ``$\preceq$'' relationship, we use $dd^T \preceq nI$ to
obtain $(2/n) \eps dd^T \preceq 2 \eps I$, $\bfone \bfone^T D^2
=n^{1/2} \bfone \crr^T$ to bound the terms with $\bfone \bfone^T D^2$
(and similarly for $D^2 \bfone \bfone^T$), $D^2 \succeq 0$, $\bfone
d^T = n^{1/2} \bfone \crr^T$ to obtain $-(1/n) \bfone d^T \preceq
n^{-1/2} \bfone \crr^T$, and $\cRR \preceq I$.
\end{proof}


\subsection{Proof of Lemma~\ref{lem:T3}}

\begin{proof}
We have
\[
 P C_P^T P^T (\bfone v^T + v \bfone^T) P C_P P^T = P C_P^T P^T \bfone
 v^T P C_P P^T + \trans,
\]
where we use $\trans$ to denote the transpose of the explicitly
stated terms. From Lemma~\ref{lem:CP} and $P^T \bfone = \bfone$, we have
\begin{subequations}
\begin{align}
  \nonumber
  & (1-\ddd)^{-2} P C_P^T P^T \bfone v^T P C_P P^T                                                                                                      \\
  \label{eq:uv81}
&=  P \left[ (I-\bfone e_1^T) + \eps (I-\bfone e_1^T) (-D_P + D_P F) +
  \ddd (F-\bfone e_2^T) + \eps^2 (\crho \bfone \crr^T + \crho \cRR) \right] \bfone v^T P\\
\nonumber
& \quad\quad  \left[ (I-e_1 \bfone^T) + \eps(-D_P + F^T D_P) (I-e_1 \bfone^T) +
  \ddd (F^T-e_2 \bfone^T) + \eps^2 (\crho \crr \bfone^T + \crho \cRR) \right] P^T
  \\
\label{eq:uv8}
                      & = P(I-\bfone e_1^T) \bfone v^T P (I-e_1 \bfone^T) P^T                                                                                               \\
  \nonumber
                      & \quad\quad 
-\eps \big\{ P(I-\bfone e_1^T) D_P (I-F) \bfone v^T P (I-e_1 \bfone^T) P^T                                                                                                  \\
  \nonumber
                      & \quad\quad\quad\quad + 
P(I-\bfone e_1^T) \bfone v^TP (I-F^T) D_P (I-e_1 \bfone^T) P^T \big\}                                                                                                       \\
  \nonumber
                      & \quad\quad +\ddd \big\{ 
P(F-\bfone e_2^T) \bfone v^TP (I-e_1 \bfone^T) P^T                                                                                                                          \\
  \nonumber
                      & \quad\quad\quad\quad +
P(I-\bfone e_1^T) \bfone v^T P (F^T - e_2 \bfone^T) P^T
\big\} \\
  \nonumber
& \quad\quad + \eps^2 n^{1/2} \|v\| (\crho \cRR + \crho (\crr \bfone^T + \bfone \crr^T) + \crho \bfone \bfone^T).
\end{align}
\end{subequations}
To derive the remainder term (the coefficient of $\eps^2$ in
\eqref{eq:uv8}), we need to consider the coefficients of $\eps^2$,
$\ddd \eps$, and $\ddd^2$ from \eqref{eq:uv81}. The coefficient of the
$\eps^2$ term is
\begin{align}
  \nonumber
  &  \eps^2 P(I-\bfone e_1^T) \bfone v^TP (\crho \crr \bfone^T + \crho \cRR) P^T \\
  \nonumber
  & \quad +
    \eps^2 P(\crho \bfone \crr^T + \crho \cRR) \bfone v^T P(I-e_1 \bfone^T) P^T \\
  \label{eq:kg5}
  & \quad + \eps^2 P(I-\bfone e_1^T) (-D_P + D_P F) \bfone v^T (-D_P + F^T D_P) (I-e_1 \bfone^T) P^T.
\end{align}
Using $(I-\bfone e_1^T) \bfone = \bfone - \bfone = 0$, we see that the
first term in this expression vanishes. From \eqref{eq:def.F} and
\eqref{eq:F.facts}, we have several other identities:
\begin{equation} \label{eq:F1en}
(I-F) \bfone = e_n, \quad
(F-\bfone e_2^T)\bfone = F \bfone - \bfone = -e_n.
\end{equation}
In the third term, we thus have that $(-D_P+D_PF) \bfone = -D_P( I-F)
\bfone = -D_P e_n = \crho \crr$.  We also have that $e_1^T D_P e_n =0$
and $v^T (-D_P+F^TD_P) = \crho \|v\| \crr^T$. Thus \eqref{eq:kg5}
becomes
\begin{align*}
  & \eps^2 P(\crho (\crr^T\bfone) \bfone + \crho \cRR \bfone) v^T (I-Pe_1 \bfone^T) \\
  & \quad 
  + \crho \eps^2 \|v\| P(I-\bfone e_1^T) (-D_P e_n) \crr^T (I-e_1 \bfone^T)P^T  \\
  &= \eps^2 (\crho n^{1/2} \bfone + \crho n^{1/2} \crr)(v^T - \|v\| \crho \bfone^T) \\
  & \quad +
  \crho \eps^2  \|v\| \crr \crr^T (I-e_1 \bfone^T) P^T \\
  &= \eps^2 n^{1/2} \|v\| (\crho \cRR + \crho(\bfone \crr^T + \crr \bfone^T) + \crho \bfone \bfone^T),
\end{align*}
which is accounted for by the $\eps^2$ term in \eqref{eq:uv8}.  We
turn next to the coefficient of $\eps\ddd$ in \eqref{eq:uv81}. This
term consists of the following expression plus its transpose:
\begin{align*}
  & \ddd \eps P (I-\bfone e_1^T) (-D_P) (I-F) \bfone v^T P (F^T-e_2 \bfone^T) P^T \\
  &= \ddd\eps (P- \bfone e_1^T)(-D_P) e_n   v^T P (F^T-e_2 \bfone^T) P^T \quad \mbox{from \eqref{eq:F1en}} \\
  &= \ddd\eps (PD_P e_n) v^T P (F^T-e_2 \bfone^T) P^T  \quad \mbox{since $e_1^T D_P e_n=0$} \\
  &= \ddd \eps \crr (\crho \|v\| \crr^T - \crho \|v\| \bfone^T).
\end{align*}
Because $\ddd \le \eps$, this term (plus its transpose) can also be
accounted for by the $\eps^2$ term in \eqref{eq:uv8}. For the
coefficient of $\ddd^2$ in \eqref{eq:uv81}, we have, using
\eqref{eq:F1en} again,
\begin{align*}
  &\ddd^2 P(F-\bfone e_2^T) \bfone v^T P (F^T - e_2 \bfone^T) P^T \\
  &=-\ddd^2 Pe_n  v^T P (F^T - e_2 \bfone^T) P^T 
  = \ddd^2 \crr \|v\| (\crho \crr^T + \crho \bfone^T) 
  = \ddd^2 \|v\| (\crho \cRR + \crho \crr \bfone^T),
\end{align*}
which can also be absorbed into the $\eps^2$ term in \eqref{eq:uv8}.

Returning to the lower-order terms in \eqref{eq:uv8}, we use again
the fact that $(I-\bfone e_1^T) \bfone = 0$ to eliminate the $O(1)$
term, and also one of the two terms in the coefficients of both $\eps$
and $\ddd$.  We thus obtain
\begin{align}
\nonumber
                      & (1-\ddd)^{-2} P C_P^T P^T \bfone v^T P C_P P^T                                                                                                      \\ 
\nonumber
& = -\eps \big\{ P(I-\bfone e_1^T) D_P (I-F) \bfone v^T P (I-e_1 \bfone^T) P^T \big\}                                                                 \\
\nonumber
& \quad\quad + \ddd \big\{  P(F-\bfone e_2^T) \bfone v^TP (I-e_1 \bfone^T) P^T \big\} \\
\label{eq:T3.2}
& \quad\quad + \eps^2 n^{1/2} \|v\| (\crho \cRR + \crho (\crr \bfone^T + \bfone \crr^T) + \crho \bfone \bfone^T).
\end{align}
Additionally, we have from \eqref{eq:def.DE} that 
\begin{align*}
P(I-e_1 \bfone^T) P^T & = I - (Pe_1) \bfone^T                                                                                                                               \\
P(I-\bfone e_1^T) D_P & = P (I-\bfone e_1^T)P^TDP =
(I-\bfone (Pe_1)^T) DP.
\end{align*}
By substituting these identities into \eqref{eq:T3.2}, we obtain
\begin{align}
\nonumber
                      & (1-\ddd)^{-2} P C_P^T P^T \bfone v^T P C_P P^T                                                                                                      \\
\nonumber
                      & =-\eps \big\{ (I-\bfone (Pe_1)^T)DP e_n v^T (I-(Pe_1) \bfone^T) \big\}                                                                              \\
\nonumber
& \quad\quad  -\ddd \big\{ (Pe_n) v^T (I-(Pe_1) \bfone^T) \big\} \\
\nonumber
& \quad\quad + \eps^2 n^{1/2} \|v\| (\crho \cRR + \crho(\bfone \crr^T + \crr \bfone^T) + \crho \bfone \bfone^T)    \\
\nonumber
                      & = - \eps \big\{ D (Pe_n) v^T (I-(Pe_1) \bfone^T) \big\} -
\ddd \big\{ (Pe_n) v^T (I-(Pe_1) \bfone^T) \big\} \\
\label{eq:T3.3}
& \quad\quad + \eps^2 n^{1/2} \|v\| (\crho \cRR + \crho(\bfone \crr^T + \crr \bfone^T) + \crho \bfone \bfone^T)
\end{align}
where the second equality follows from 
\[
(Pe_1)^TD(Pe_n) = e_1^T D_P e_n=0,
\]
since $D_P$ is diagonal for all $P$.

Taking expectations, we have for the coefficient of $(-\ddd)$ in
\eqref{eq:T3.3} that 
\begin{align}
\nonumber
\E_P \, \big( (Pe_n) v^T (I-(Pe_1) \bfone^T) \big) & =
\E_P \, (Pe_n)v^T - \big[ \E_P \, (Pe_n) (v^T Pe_1) \big] \bfone^T \\
\nonumber
 & = \frac{1}{n} \bfone v^T - \frac{1}{n(n-1)} \left( \sum_{\stackrel{i,j=1}{i \ne j}}^n v_i e_j \right) \bfone^T            \\
\nonumber
 & = \frac{1}{n} \bfone v^T - \frac{1}{n(n-1)} \left( \sum_{i=1}^n v_i \sum_{\stackrel{j=1}{i \ne j}}^n e_j \right) \bfone^T \\
\nonumber
 & = \frac{1}{n} \bfone v^T - \frac{1}{n(n-1)} \left( \sum_{i=1}^n v_i (\bfone - e_i)  \right) \bfone^T                      \\
\nonumber
 & = \frac{1}{n} \bfone v^T - \frac{1}{n(n-1)} \left( (\bfone^Tv) \bfone \bfone^T  - v \bfone^T)  \right)                    \\ 
\label{eq:T3.4}
 & = \frac{1}{n} \bfone v^T - \frac{(\bfone^Tv)}{n(n-1)} \bfone \bfone^T + \frac{1}{n(n-1)} v \bfone^T,
\end{align}
where the second equality is from a conditional expectation over
permutation matrices $P$ such that $Pe_1 = j$ and $Pe_n=i$, for all
$i,j=1,2,\dotsc,n$ with $i \ne j$.  By combining \eqref{eq:T3.4} with
its transpose, we obtain the full coefficient of $(-\ddd)$ in
\eqref{eq:T3.0}, which is
\[
\left( \frac{1}{n} + \frac{1}{n(n-1)} \right) (\bfone v^T + v \bfone^T) -
\frac{2 (\bfone^Tv)}{n(n-1)}  \bfone \bfone^T =
\frac{1}{n-1} (\bfone v^T + v \bfone^T) - \frac{2 (\bfone^Tv)}{n(n-1)}
\bfone \bfone^T.
\]
This verifies the $O(\ddd)$ term in \eqref{eq:T3.0}. 

We note that the coefficient of $(-\eps)$ in \eqref{eq:T3.3} is the
same as the coefficient of $(-\ddd)$, except for being multiplied from
the left by $D$, which is independent of $P$. Thus the expectation of
this term is simply \eqref{eq:T3.4}, multiplied from the left by $D$,
that is,
\[
\frac{1}{n} D \bfone v^T - \frac{\bfone^Tv}{n(n-1)} (D \bfone) \bfone^T +
\frac{1}{n(n-1)} Dv\bfone^T =
\frac{1}{n} dv^T - \frac{\bfone^Tv}{n(n-1)} d \bfone^T +
\frac{1}{n(n-1)} Dv\bfone^T.
\]
We obtain the full coefficient of $(-\eps)$ in \eqref{eq:T3.0} by
adding this quantity to its transpose, to obtain
\[
\frac{1}{n} (dv^T + vd^T) - \frac{\bfone^Tv}{n(n-1)} (d \bfone^T + \bfone d^T )
+ \frac{1}{n(n-1)} (Dv \bfone^T + \bfone v^TD),
\]
as required.

For the ``$\preceq$''
result \eqref{eq:T3.0.prec}, we use
\[
\| dv^T \| = n^{1/2} \crho \|v\|, \quad
\|d\|  \le n^{1/2}, \quad
|\bfone^Tv| \le n^{1/2} \|v\|, \quad
\|Dv\| \le \|d\|, 
\]
together with $\ddd \le \eps$ from \eqref{eq:epsdd}. We also use $\crr
\bfone^T + \bfone \crr^T = n^{1/2} \cRR$ and $-I \preceq \cRR \preceq
I$ for symmetric $\cRR$ to absorb the $\eps^2 (\crr \bfone^T + \bfone
\crr^T)$ term in \eqref{eq:T3.0} into the $I$ term in
\eqref{eq:T3.0.prec}.


For \eqref{eq:T3.v1}, we have
\begin{align}
  \nonumber
  & (1-\ddd)^{-1} P C_P^T P^T \bfone \\
  \nonumber
  &= (1-\ddd)^{-1} PC_P^T \bfone \\
  \nonumber
  &= P \left[ (I-\bfone e_1^T) + \eps (I-\bfone e_1^T)D_P(-I+  F) + \ddd (F-\bfone e_2^T) + \eps^2 (\crho \bfone \crr^T + \crho \cRR) \right] \bfone \\
  \nonumber
  &= P \left[ -\eps (I-\bfone e_1^T) D_P e_n - \ddd e_n
    + \eps^2 n^{1/2} (\crho \bfone + \crho \crr) \right] \\
  \label{eq:zh6}
  &= P \left[ -\eps D_P e_n - \ddd e_n +  \eps^2 n^{1/2} (\crho \bfone  + \crho \crr) \right],
\end{align}
where we used the following identities for the third equality:
\begin{align*}
  & (I-\bfone e_1^T)\bfone = 0, \quad (-I+F) \bfone = -e_n, \quad (F-\bfone e_2^T)\bfone = -e_n, \\
  & \crr^T\bfone  \le n^{1/2}, \quad \cRR \bfone = \crho n^{1/2} \crr,
\end{align*}
and $e_1^TD_Pe_n=0$ for the fourth equality. By substituting $D_P =
P^TDP$ into \eqref{eq:zh6}, we obtain
\[
(1-\ddd)^{-1} PC_PP^T \bfone = -\eps DPe_n - \ddd Pe_n + \eps^2 n^{1/2} (\crho \bfone + \crho \crr).
\]
By taking the outer product of this vector with itself, and using
$\ddd \le \eps$, we obtain
\begin{align*}
  &   (1-\ddd)^{-2} PC_P^TP^T \bfone\bfone^T P C_P P^T \\
  &=
  \eps^2 D(Pe_n)(Pe_n)^T D + \ddd \eps [D(Pe_n)(Pe_n)^T + (Pe_n) (Pe_n)^TD] + \ddd^2 (Pe_n)(Pe_n)^T \\
  & \quad\quad + \eps^3 n^{1/2} (\crho (\bfone \crr^T + \crr \bfone^T) + \crho \cRR),
\end{align*}
where in the remainder term we used $\|Pe_n\| = 1$ and $\|D\| \le
1$. By taking expectations over $P$, and using $\E_P (Pe_n)(Pe_n)^T =
n^{-1} I$, we obtain
\begin{align*}
  &   (1-\ddd)^{-2} \E_P (PC_P^T P^T \bfone \bfone^T PC_P P^T)  \\
  &= n^{-1} \eps^2 D^2 + 2n^{-1} \eps \ddd D + n^{-1} \ddd^2 I + \eps^3 n^{1/2} (\crho(\bfone \crr^T + \crr \bfone^T) + \crho \cRR) \\
  &= \crho \eps^2 n^{-1} \cRR + \crho \eps^3 n \cRR = \crho \eps^2 \cRR,
\end{align*}
where we used \eqref{eq:epsdd} in the last expression to deduce that
$\eps^3 n \le \eps^2$.

\end{proof}

\section{Proof of Lemma~\ref{lem:hatbar}} \label{app:C}

\begin{proof}

  Note first that $\hat\eta_t$, $\hat\nu_t$, $\hat\eps_t$, and
  $\hat\tau_t$ are all nonnegative by definition. In this proof, we
  use repeatedly that they can be bounded by $| \tilde\eta_t|$, $|
  \tilde\nu_t|$, $| \tilde\eps_t|$, and $|\tilde\tau_t|$, respectively,
  though the $|\cdot|$ are unnecessary in the case of $\tilde\eps_t$
  (since its exact value can be determined trivially from
  \eqref{eq:4R}) and in the case of $\tilde\tau_t$ (which can be
  assumed without loss of generality to be nonnegative, as mentioned
  in the proof of Theorem~\ref{th:4R}).

The proof is by induction on $t$. We show first that the bounds
\eqref{eq:hs6} hold for $t=1$.

We have from \eqref{eq:1.8} and the obvious property $\hat\eps_t =
(1-\ddd)^{2t} \eps$ from \eqref{eq:4R} that
\[
\hat\eps_1 = (1-\ddd)^2 \eps \le (1-1.8\ddd) \eps = \bar\eps_1,
\]
verifying \eqref{eq:hs6b} for $t=1$. For \eqref{eq:hs6a}, we have from
\eqref{eq:4R} with $t=0$, using the initial values \eqref{eq:4Rinit} and
the bounds in \eqref{eq:hs2} that
that
\begin{align*}
  (1-\ddd)^{-2}  \hat\eta_1 \le (1-\ddd)^{-2} | \tilde\eta_1| & \le (1+\rhobar \eps^2) \ddd + \rhobar\eps^2 (1-\ddd) +(2\eps + \rhobar \eps^2) \eps \\
 & = \ddd + \left( 2 + \rhobar + \rhobar \eps  \right) \eps^2\\
  & \le \ddd + \rhohat \eps^2 \le 3 \ddd.
\end{align*}
It follows from $\rhohat \ge 3$ and \swmodify{$\ddd \le .2$ (see
  \eqref{eq:hs2.2})} that
\[
\hat\eta_1 \le 3(1-\ddd)^2 \ddd \le 3(1-1.4\ddd) \ddd \le 1.5 \rhohat (1-1.4\ddd) \ddd = \bar\eta_1,
\]
verifying \eqref{eq:hs6a} for $t=1$. For \eqref{eq:hs6c} with $t=1$, we have
\begin{align*}
  \hat\tau_1 = \tilde\tau_1  & \le (1-\ddd)^2 \rhobar \eps n^{-1/2} \ddd + (1-\ddd)^2 ( \rhobar n^{-1/2} + \rhobar \eps^2) \eps \\
  & \le (1-1.4\ddd) (.5) \rhobar \eps \ddd + (1-1.8\ddd) (.5 \rhobar + .04 \rhobar) \eps,
\end{align*}
where for the second inequality we used $n^{-1/2} \le .5$ and $\eps^2
\le .04$. Continuing, we use $\bar\eta_1 \ge 4 (1-1.4\ddd) \ddd$ and
$\bar\eps_1 = (1-1.8\ddd)\eps$ to write
\begin{equation} \label{eq:hs8}
  \hat\tau_1 \le \frac18 \rhobar\eps \bar\eta_1 + .54 \rhobar \bar\eps_1,
\end{equation}
which suffices to prove \eqref{eq:hs6c} for $t=1$. For
\eqref{eq:hs6d}, we simply use $\eps \le .2$ from \eqref{eq:hs2}.

For \eqref{eq:hs6e} with $t=1$, we have from \eqref{eq:4R} and
\eqref{eq:4Rinit}, using again $\eps^2 \le .04$ from \eqref{eq:hs2},
as well as $\dav \le 1$ from \eqref{eq:Dd} that
\begin{align*}
   \hat\nu_t \le | \tilde\nu_t| & \le (1-\ddd)^2 (1+\rhobar\eps^2) \ddd + (1-\ddd)^2 (\dav + \rhobar \eps^2) \eps \\
  & \le (1-\ddd)^2 (1+\rhobar\eps^2) \ddd + (1-\ddd)^2 (1+.04\rhobar) \eps \\
  & \le (1-1.4 \ddd) \rhohat \ddd + (1-1.8\ddd) (1+.04 \rhobar)\eps \\
  & \le \bar\eta_1 + (1+.04\rhobar) \bar\eps_1,
\end{align*}
which suffices to demonstrate \eqref{eq:hs6e} for $t=1$.

Assuming now that \eqref{eq:hs6} holds for some $t \ge 1$, we prove
that the bounds holds for $t+1$ as well. We start with
\eqref{eq:hs6c}. It follows from \eqref{eq:4R} that
\begin{align*}
  (1-\ddd)^{-2} \hat\tau_{t+1}  \le (1-\ddd)^{-2} | \tilde\tau_{t+1} | & \le
  \rhobar \eps n^{-1/2} | \hat\eta_t| + (\rhobar n^{-1/2} + \rhobar \eps^2) | \hat\eps_t| \\
  & \le .5 \rhobar \eps \bar\eta_t + (.5 \rhobar + .04 \rhobar) \bar\eps_t \\
  & \le .5 \rhobar \eps \bar\eta_t + .54 \rhobar \bar\eps_t.
\end{align*}
It then follows from \eqref{eq:hs5} that
\[
\hat\tau_{t+1} \le .5 \rhobar \eps \bar\eta_{t+1} + .54 \rhobar \bar\eps_{t+1},
\]
as required. As earlier, \eqref{eq:hs6d} follows immediately when we
note that $\eps \le .2$, from \eqref{eq:hs2}.

For \eqref{eq:hs6e}, we have
\begin{align*}
  (1-\ddd)^{-2}  \hat\nu_{t+1} &\le
  (1+\rhobar \eps^2) \bar\eta_t + (\dav + \rhobar \eps^2) \bar\eps_t + \rhobar\eps n^{-1/2} | \hat\tau_t| \\
  & \le (1+\rhobar\eps^2+(\rhobar\eps n^{-1/2})(.1) \rhobar) \bar\eta_t +
  (\dav + \rhobar\eps^2 + (\rhobar \eps n^{-1/2}) (.54) \rhobar) \bar\eps_t,
\end{align*}
where we used \eqref{eq:hs6d} for the second inequality. Using now the
bounds $\rhobar \eps^2 \le .05$ and $n^{-1/2} \le .5$ (from
\eqref{eq:hs2}), $\dav \le 1$, and $\eps n^{-1/2} = (\eps n) n^{-3/2}
\le n^{-3/2} \le .1$, we have
\[
(1-\ddd)^{-2}  \hat\nu_{t+1} \le (1.1 + .01 \rhobar^2) \bar\eta_t + (1.1 + .1 \rhobar^2) \bar\eps_t,
\]
so that
\[
\hat\nu_{t+1} \le  (1.1 + .01 \rhobar^2) \bar\eta_{t+1} + (1.1 + .1 \rhobar^2) \bar\eps_{t+1},
\]
as required.

The proof for \eqref{eq:hs6b} is trivial, since
\[
\hat\eps_{t+1} = (1-\ddd)^{2(t+1)} \eps = (1-\ddd)^2 \hat\eps_t \le
(1-1.8 \ddd) \hat\eps_t \le (1-1.8\ddd) \bar\eps_t = \bar\eps_{t+1}.
\]

We now prove \eqref{eq:hs6a} for $t$ replaced by $t+1$. We have,
substituting from the other formulas in \eqref{eq:hs6}, and using the
bounds in \eqref{eq:hs2}, that
\begin{align*}
  (1-\ddd)^{-2} \hat\eta_{t+1} & \le
  (1+\rhobar\eps^2) \bar\eta_t + \rhobar\eps^2 \bar\nu_t + (2\eps + \rhobar \eps^2) \bar\eps_t + \rhobar \eps | \hat\tau_t| \\
  & \le (1+\rhobar\eps^2) \bar\eta_t + \rhobar\eps^2 [(1.1+.01\rhobar^2) \bar\eta_t + (1.1+.1 \rhobar^2) \bar\eps_t] \\
  & \quad\quad + (2\eps+\rhobar\eps^2) \bar\eps_t + \rhobar\eps [ .5 \eps \rhobar \bar\eta_t+.54 \rhobar \bar\eps_t] \\
  & \le [ 1 + \rhobar \eps^2 + \rhobar\eps^2(1.1+.01 \rhobar^2) + .5 \rhobar^2 \eps^2] \bar\eta_t \\
  & \quad\quad + [\rhobar \eps^2 (1.1+.1\rhobar^2) + 2\eps + \rhobar\eps^2 + .54 \rhobar^2 \eps] \bar\eps_t \\
  & \le [1+\eps^2 (\rhobar + \rhobar(1.1+.01 \rhobar^2) + .5 \rhobar^2)] \bar\eta_t \\
  & \quad\quad + [.5\eps(1.1+.1 \rhobar^2) + 2\eps + .5\eps + .54 \rhobar^2 \eps] \bar\eps_t \\
  & \le [1+\eps^2 (2.1 \rhobar + .5 \rhobar^2 + .01 \rhobar^3)]\bar\eta_t + \eps [.55+.05\rhobar^2 + 2.5 + .54 \rhobar^2] \bar\eps_t \\
  & \le (1+\rhohat \eps^2) \bar\eta_t + \rhohat \eps \bar\eps_t,
  \end{align*}
where we used the definition \eqref{eq:rhohat} of $\rhohat$ for the
final inequality. Thus from \eqref{eq:148}, substituting from
\eqref{eq:etaeps}, and using $\eps^2 < \ddd$ from \eqref{eq:epsdd}, we
have
\begin{align*}
  \hat\eta_{t+1} & \le (1-\ddd)^2 (1+\rhohat\eps^2) \bar\eta_t
  + (1-\ddd)^2 \rhohat \eps \bar\eps_t \\
  & \le (1-1.4 \ddd) \bar\eta_t + (1-1.8 \ddd) \rhohat \eps \bar\eps_t \\
  & \le 1.5 \rhohat (1-1.4\ddd)^{t+1} t \ddd + (1-1.8\ddd)^{t+1} \rhohat \eps^2 \\
  & \le 1.5 \rhohat (1-1.4\ddd)^{t+1} t \ddd + (1-1.4\ddd)^{t+1} \rhohat \eps^2 \\
  & \le (1-1.4\ddd)^{t+1} (1.5\rhohat t \ddd + \eps^2) \\
  & \le (1-1.4\ddd)^{t+1} (1.5\rhohat t \ddd + \ddd) \\
  & \le (1-1.4\ddd)^{t+1} (1.5) \rhohat (t+1) \ddd  = \bar\eta_{t+1},
\end{align*}
as required. This completes the inductive step and hence the proof.
  \end{proof}

\end{document}